\documentclass{amsart}

\usepackage{latexsym}
\usepackage{amssymb}
\usepackage{amsmath}
\usepackage{amsfonts}
\usepackage[mathscr]{eucal}
\usepackage[arrow, matrix, curve]{xy}

\usepackage{amscd}

\usepackage{tikz-cd}

\usepackage{epsfig}

\usepackage{amsthm}
\usepackage{hyperref}
\usepackage{comment}

\newtheorem{thm}{Theorem}[section]
\newtheorem{cor}[thm]{Corollary}
\newtheorem{lemma}[thm]{Lemma}
\newtheorem{prop}[thm]{Proposition}
\newtheorem{sublem}[thm]{Sublemma}

\theoremstyle{definition}
\newtheorem{definition}[thm]{Definition}

\newtheorem{example}[thm]{Example}

\theoremstyle{remark}

\newtheorem*{ack}{Acknowledgments}

\numberwithin{equation}{section}

\newcommand{\sphere}{\mathrm{\mathbb{S}}}
\newcommand{\crcl}{\mathsf{S}^1}

\newcommand{\Fix}{\mathrm{Fix}}
\newcommand{\G}{\mathsf G}
\newcommand{\T}{\mathsf T}
\newcommand{\RR }{\mathbb R}

\newcommand{\NN}{\mathbb N}
\newcommand{\ZZ}{\mathbb Z}

\begin{document}
\title[Nonnegatively curved quotient spaces with boundary]{Nonnegatively curved quotient spaces with boundary}

\author[W.\ Spindeler]{Wolfgang Spindeler$^*$}

\thanks{$^*$ All results were obtained when the author was part of SFB 878 - Groups, Geometry \& Actions at the University of M\"unster}

\email{wolfgang.spindeler@gmail.com}





\begin{abstract}
Let $M$ be a compact nonnegatively curved Riemannian manifold admitting an isometric action by a compact Lie group $\G$ in a way that the quotient space $M/\G$ has nonempty boundary. Let $\pi : M \to M/\G$ denote the quotient map and $B$ be any boundary stratum of $M/\G$. Via a specific soul construction for $M/\G$ we construct a smooth closed submanifold $N$ of $M$ such that $M \setminus \pi^{-1}(B)$ is diffeomorphic to the normal bundle of $N$. As an application we show that a simply connected torus manifold admitting an invariant metric of nonnegative curvature is rationally elliptic.
\end{abstract}

\maketitle

\section{introduction}
In order to find possible obstructions to positive and nonnegative curvature on Riemannian manifolds it was suggested by Grove in the early 90's to study positively and nonnegatively curved Riemannian manifolds with a large amount of symmetry. A lot of results have been obtained in this area and it is very active until today. For an introduction see for example \cite{grove02}. While considerably more structure results have been obtained in the case of positive curvature, here we present a new tool for the study of nonnegatively curved Riemannian manifolds admitting certain symmetries. More precisely we consider isometric actions on compact and connected nonnegatively curved Riemannian manifolds in a way that the quotient space has nonempty boundary. Our main result is the following theorem.

\begin{thm}\label{main}
Let $(M,g)$ be a compact, connected and nonnegatively curved Riemannian manifold admitting an isometric action by a compact Lie group $\G$ in a way that the quotient space $M/\G$ has nonempty boundary $\partial M/\G$. Let $\pi : M \to M/\G$ denote the quotient map and $B$ be an arbitrary boundary stratum of $M/\G$. Then there exists a closed smooth $\G$-invariant submanifold $N$ of $M$ such that $M \setminus \pi^{-1}(B)$ is equivariantly diffeomorphic to the normal bundle of $N$.
\end{thm}
The simplest stratum of $\partial M/\G$ is given by the boundary itself. For details on the definition we refer to section \ref{extremal}.

Among others, initial research on positively curved manifolds with symmetry was done by Grove and Searle. In \cite{grove-searle94} they obtain restrictions on the symmetry rank of such manifolds. The main ideas for our results already arise there and more self contained in \cite{grove-searle97}, in particular in the form of their soul lemma 1.9. A slight generalization of this was used later by Wilking in \cite{wilking06}; given a positively curved Riemannian manifold with an isometric $\G$-action in a way that the boundary $\partial M/\G$ of the quotient space is nonempty he obtains a diffeomorphism
\begin{align}\label{soul orbit}
M \setminus \pi^{-1}(B) \cong \nu(\G \ast p),
\end{align}
where $B$ is an arbitrary boundary stratum of $M/\G$ and $\nu (\G \ast p)$ denotes the normal bundle of the orbit $\G \ast p$. Our theorem \ref{main} can be seen as the best possible generalization of this result to the case of nonnegative curvature. 

Note that given the situation of theorem \ref{main}, but complementary assuming that $\partial M/\G$ is empty, and denoting by $\mathsf H$ the principal isotropy group of the action, it follows from \cite{wilking07dual} that for $\mathsf H \neq \{1\}$ there exists a subgroup $\mathsf K \subseteq \G$ with $\mathsf H \subseteq \mathsf K \subseteq N(\mathsf H)$ and a $\mathsf K$-invariant metric on $\G/\mathsf K$ together with an equivariant Riemannian submersion $M \to \G / \mathsf K$ with totally geodesic fibers. This also suggests that one should quite often be in the situation of our theorem when dealing with isometric actions in nonnegative curvature.

The results by Grove and Searle as well as Wilking were used as tools to obtain classification results for positively curved manifolds admitting certain symmetries. Theorem \ref{main} is propably useful for the classification program of nonnegatively curved manifolds with symmetries as well. As an indication of its potential applications we obtain the following result about torus manifolds. 
\begin{thm}\label{torus}
Let $M$ be a compact and simply connected torus manifold admitting an invariant metric of nonnegative curvature. Then $M$ is rationally elliptic.
\end{thm}
By definition a torus manifold is a compact, connected and orientable manifold of dimension $2n$ admitting a smooth and effective action by the $n$-dimensional torus in way that its fixed point set is nonempty. Theorem \ref{torus} was originally obtained in the authors thesis \cite{spindeler14}. It was deduced there from results on nonnegatively curved fixed point homogeneous manifolds. These results, and therefore theorem \ref{torus}, follow from our theorem \ref{main} and will be discussed in section \ref{section fph}.

The proof of theorem \ref{main} will be carried out in sections \ref{section preleminaries} and \ref{section proof main}. To give an overview of the arguments used let us first sketch how one obtains the decomposition \eqref{soul orbit} in the positively curved case: Given a compact and positively curved Riemannian manifold $M$ equipped with an isometric $\G$-action, the quotient space $M/\G$ is a positively curved Alexandrov space. Further, if $\partial M/\G$ is nonempty, the distance function $d_B$ to a boundary stratum $B \subseteq M/\G$ is a strictly concave function on $(M/\G) \setminus B$. Therefore, there exists a unique orbit $\G \ast p$ at maximal distance to $B$. Also by concavity the distance function to the orbit $\G \ast p$ is noncritcal on $M \setminus (\pi^{-1}(B) \cup \G \ast p)$. The result then follows from standard arguments in critical point theory for distance functions.

Now in the case of nonnegative curvature most of this arguments carry over. The main difference arises, since the distance function $d_B$ is only concave and therefore the set $C$ of maximal distance to $B$ does not consist of only a single orbit. In fact simple examples show that it is possible that $\pi^{-1}(C) \subset M$ is not a smooth submanifold without boundary. However, since $C$ is the super level set of a concave function, its geometry is quite rigid. Most importantly for our arguments, $C$ is convex with respect to projections of horizontal geodesics of $M$. For our proof we first derive basic geometric and regularity properties of convex subsets of quotient spaces and in particular super level sets of concave functions. Then via a specific soul construction for $M/\G$ we construct a convex subset whose preimage $N \subset M$ is a smooth submanifold without boundary in a way that we can additionally control the regularity of the distance function to $N$. The arguments for this are quite complicated but at the same time mostly of an elementary character.

\begin{ack}
I am grateful to Burkhard Wilking for his support during the work on my thesis where the techniques developed here originate.
\end{ack}


\section{preleminaries}\label{section preleminaries}
Throughout this section let $M = (M,g)$ be a connected Riemannian manifold (not necessarily complete) equipped with an isometric action by a compact Lie group $\G$. The quotient space $M/\G$ is denoted $M^*$. $g$ induces a length metric $d$ on $M^*$ via the distance of orbits. The quotient map is denoted 
$$\pi : M \to M^*.$$
If $M$ has a lower bound on its sectional curvatures on an open set $U$ invariant under $\G$ then $U/\G \subset M^*$ has the same lower curvature bound in the distance comparison sense. In particular if $M$ is complete and has a uniform lower curvature bound the quotient space $M^*$ is an Alexandrov space with the same lower curvature bound as $M$.
For a point $p \in M$ with orbit $\G \ast p$ and isotropy group $\G_p$ the tangent space $T_pM$ decomposes $\G_p$-invariant into the normal space to the orbit $\G \ast p$ at $p$ and its tangent space at $p$;
$$T_pM = N_p(\G \ast p) \oplus T_p(\G \ast p).$$
 For $x \in M^*$ the tangent cone at $x$ is denoted by $T_xM^*$ and the space of directions at $x \in M^*$ is denoted by $\Sigma_xM^*$. A continuous curve $\gamma : I \to M^*$ (or more generally mapping into an Alexandrov space) is called a geodesic if $\gamma$ is locally a minimal segment between the points on it. Then $\Sigma_xM^*$ consists of the initial directions of geodesics emanating from $x$. By the slice theorem it follows that for $p \in M$ with $\pi(p) = x$ we have  $T_xM^* = N_p(\G \ast p)/\G_p$ up to isometry. Thus for all $p \in M$ the differential of $\pi$ at $p$ can be identified with the quotient map
$$d\pi_p : N_p(\G \ast p) \to N_p(\G \ast p)/\G_p = T_xM^*.$$
By the type of $x \in M^*$ we mean the orbit type of any $p \in M$ with $\pi(p) = x$. A point in $M^*$ of maximal type is called regular. A more detailed introduction to the geometry of orbit spaces can be found for example in \cite{grove02}.
\\

\subsection{convex subsets of quotient spaces}
The aim of this section is to obtain regularity results for a convex subset $C \subseteq M^*$ and more importantly its preimage $\pi^{-1}(C) \subseteq M$. This is motivated by the result of Cheeger and Gromoll in \cite{cheeger-gromoll72} that a closed and convex subset of a Riemannian manifold is a smooth totally geodesic submanifold possibly with nonsmooth boundary. As mentioned in the introduction convex sets are the basic objects to study for the proof of our main theorem, since they arise as the level sets of the distance function to a boundary stratum of $M^*$.
\begin{definition}
 Let $A$ be an Alexandrov space and $C \subseteq A$. Then $C$ is convex if for all $x,y \in C$ there exists a minimal geodesic from $x$ to $y$ that is contained in $C$. $C$ is locally convex if $C$ is connected and for all $x \in C$ there exists $\epsilon > 0$ such that for all $y,z \in B_\epsilon(x)$ there exists a minimal geodesic from $y$ to $z$ that is contained in $C$.
\end{definition} 
There are several possible natural notions of convexity. We chose this definition for our first basic observations, since it is the weakest one that comes to mind. Later on we will consider stronger convexity properties aiming at the proof of theorem \ref{main}. The connectedness assumption in the local version guarantees that local properties, as for example the dimension, are defined globally.

It is easy to see that a closed and locally convex subset $C$ of an Alexandrov space equipped with the induced intrinsic metric is again an Alexandrov space with the same lower curvature bound. Also a geodesic of $C$ is a geodesic of $A$ as well. Similarly note that a closed subset $C \subseteq A$ is convex if and only if the induced metric on $C$ is intrinsic and a geodesic of $C$ is also a geodesic of $A$.

The following simple example shows that the preimage $\pi^{-1}(C) \subset M$ of a closed and convex subset $C \subset M^*$ might fail to be a topological manifold.
\begin{example}\label{example 1}
Let $\mathbb Z_k$ act on $\RR^2$ via rotation around the origin by an angle $\frac{2 \pi} k$. Then the quotient space is isometric to the cone over a circle of length $\frac{2 \pi} k$. A ray emanating from the tip of the cone defines a convex subset $C$ and the preimage of $C$ in $\RR^2$ is isometric to $k$ copies of $[0,\infty[$ glued together at $0$.
\end{example}
A bit more complicated example shows that similar effects occur also for convex subsets with empty boundary (by boundary we always mean the intrinsic boundary of $C$ considered as an Alexandrov space if not mentioned otherwise).
\begin{example}\label{example 2}
Let the circle $\crcl$ act on the $3$-sphere $\sphere^3$ via the Hopf action. Then the quotient space is given by $\sphere^2(1/2)$, the $2$-dimensional sphere of radius $1/2$. Suspending the Hopf action we obtain an isometric action of $\crcl$ on $M = \sphere^4$ whose quotient space $M^*$ is the spherical suspension of $\sphere^2(1/2)$. Consider a great circle $\Gamma$ in $\sphere^2(1/2)$ and let $C \subset M^*$ denote the subset obtained by suspending $\Gamma$. Then $C$ is closed and convex in $M^*$ and has empty boundary. The preimage $\pi^{-1}(C) \subset \sphere^4$ is given by the spherical suspension of $\pi^{-1}(\Gamma)$. Since $\pi^{-1}(\Gamma)$ is an embedded $2$-torus in $\sphere^3$, it follows that $\pi^{-1}(C) \subset \sphere^4$ is given by the spherical suspension of a torus.
\end{example}
As a first structure result we observe that a convex subset of $M^*$  has a maximal orbit type with properties analogous to the maximal type of $M$.
\begin{lemma}\label{maximal type}
 Let $C \subseteq M^*$ be locally convex. Then there exists a unique maximal type in $C$ and the points of this maximal type form an open convex and dense subset of $C$.
\end{lemma}
\begin{proof}
Since the type is constant along an open geodesic arc of $M^*$ and can only decrease in its closure, it is enough to find an open subset $U$ of $C$ on which the type is constant. The existence of such a $U$ then follows from the slice theorem and convexity. The details are left to the reader.
\end{proof}
Keeping lemma \ref{maximal type} in mind the following definition makes sense.
\begin{definition}
 Let $C \subseteq M^*$ be locally convex. A point $x \in C$ of maximal type is called a  regular point of $C$.
\end{definition}
Considering $M^*$ as a convex subset of itself this terminology is consistent with the notion of a regular point of $M^*$.  However, note that a regular point of a convex subset $C \subset M^*$ can possibly be contained in $\partial C$ in contrast to the case of $M^*$.
\begin{lemma}\label{regular point}
 Let $C \subseteq M^*$ be locally convex and locally closed (i.e. for every point $x \in C$ there exists $\epsilon > 0$ such that $B_\epsilon(x) \cap C$ is closed in $B_\epsilon(x)$). Given a regular point $x$ of $C$ there exists an open neighborhood $U$ of $x$ in $C$ such that $\hat U := \pi^{-1}(U)$ is a smooth submanifold of $M$, possibly wit nonsmooth boundary. Furhter, if $p \in M$ with $\pi(p) = x$, then $p$ is a boundary point of $\hat U$ if and only if $x$ is a boundary point of $C$.
\end{lemma}
\begin{proof}
 Let $x \in C$ be regular. Then an open neighborhood $U$ of $x \in C$ contains only regular points. Let $\pi(p) = x$. Denote by $N$ the component of points of $M$ of the same type as $p$ that contains $p$. Then $N$ is a smooth $\G$-invariant submanifold of $M$ without boundary. Note that $C \cap N^*$ is a locally convex subset of $N^*$. Therefore $C \cap N^*$ is a smooth submanifold of $N^*$, possibly with nonsmooth boundary. Since $\pi_{\vert N}$ is a smooth submersion it is clear that $\pi^{-1}(C) \cap N$ is a smooth submanifold of $N$ as well, and therefore also of $M$, and $p$ is a boundary point of $\pi^{-1}(C) \cap N$ if and only if $x$ is a boundary point of $C$.
\end{proof}
To determine the geometry of $C$ near a nonregular point $x$ we need to determine the structure of the tangent cone $T_xC$. The starting point is the following observation.
\begin{lemma}
 Let $C \subseteq M^*$ be locally convex. Then $T_xC \subseteq T_xM^*$ is convex for all $x \in C$.
\end{lemma}
In fact this follows easily from the next observation about convex subsets of Alexandrov spaces.
\begin{lemma}
 Let $A$ and $A_n$, for $n \in \mathbb N$, be Alexandrov spaces with a common lower curvature bound and $C_n \subseteq A_n$ be closed and convex. Also let $C \subseteq A$ be closed such that $(A_n, C_n, x_n) \to (A, C, x)$ in the pointed Gromov-Hausdorff sense for $x_n \in A_n$ and $x \in A$. Then $C$ is convex in $A$.
\end{lemma}
\begin{proof}
 Since $C$ is the limit of Alexandrov spaces with curvature uniformly bounded from below, $C$ itself is an Alexandrov space. Let $\gamma : [0,1] \to C$ be a minimal geodesic of $C$ between $\gamma(0)$ and $\gamma(1)$. We need to show that $\gamma$ is also a minimal geodesic of $A$. Let $0 < l < 1$ be arbitrary. Then $\gamma : [0,l] \to C$ is a unique minimal geodesic between its endpoints. 
Let $a_n, b_{n,l} \in C_n$ with $a_n \to \gamma(1)$ and $b_{n,l} \to \gamma(l)$ for $n \to \infty$. Choose a minimal geodesic $\gamma_{n,l}$ of $C_n$ from $a_n$ to $b_{n,l}$. Since $\gamma_{[0,l]}$ is a unique minimal geodesic, it follows that $\gamma_{n,l}$ converges to $\gamma_{[0,l]}$. Now, since $C_n$ is convex, we conclude that $\gamma_{n,l}$ is also a minimal geodesic of $A_n$. Therefore, the limt $\gamma_{[0,l]}$ is a minimal geodesic of $A$ for all $0 < l< 1$. Hence $\gamma$ is a minimal geodesic of $A$.
\end{proof}

In the following we are dealing with the question under what conditions a geodesic of $C$ can be extended within $C$.

\begin{lemma}\label{exponentiating interior points}
 Let $C \subseteq M^*$ be locally closed and locally convex. Then for every $x \in C$ that is not contained in the closure of the set of regular boundary points of $C$ there exists $\rho > 0$ such that $\exp_x : B_\rho(0_x) \cap T_xC \to M^*$ is a homeomorphism onto its image $B_\rho(x) \cap C$.
\end{lemma}
\begin{proof}
We may assume that $\dim C \geq 1$, since otherwise the statement is trivial. Let $x \in C$ be not contained in the closure of the set of regular boundary points. By the slice theorem there exists $\rho > 0$ such that $\exp_x : B_\rho(0_x) \to M^*$ is well defined and injective and the type of $\exp_x(tv)$ is the same for all $0 < t \leq \rho$ for every fixed $v \in \Sigma_xM^*$. After possibly decreasing $\rho$ we may further assume that $B_\rho(x) \cap C$ is closed in $B_\rho(x)$ and $B_\rho(x)$ does not contain any regular boundary point of $C$. Let $V \subset \Sigma_xC$ denote the set of directions that come from geodesics of $C$ starting at $x$ and which contain a regular point of $C$. Then all the geodesics $\gamma_v : [0,\rho[ \to M^*$ with initial conditions $\gamma_v(0) = x$ and $\dot \gamma_v(0) = v$ with $v \in V$ are defined. In fact it follows that $\gamma_v(t) \in C$ for all $0 \leq t < \rho$ for all $v \in V$, since $B_\rho(x)$ intersected with the set of regular points of $C$ is a smooth Riemannian manifold without boundary by lemma \ref{regular point}. Since the set of regular points is dense in $C$, it follows that $V$ is dense in $\Sigma_xC$. Then the claim follows, since $B_\rho(x) \cap C$ is closed in $B_\rho(x)$.
\end{proof}
Let $C \subseteq M^*$ be locally convex and $p \in M$ with $x = \pi(p)\in C$. Considering the map $d\pi_p : N_p(\G \ast p) \to T_xM^*$ it is clear that a necessary condition for a neighborhood $U$ of $x$ in $C$ to lift to a smooth submanifold of $M$ without boundary is that $d\pi_p^{-1}(T_xC) \subseteq N_p(\G \ast p)$ is a linear subspace. We will see in corollary \ref{submanifold} that this condition is also sufficient if it holds at every $q \in U$. This motivates the the following definition.
\begin{definition}
Let $x \in M^*$ and $W \subseteq T_xM^*$. Then $W$ is called linear if $d\pi_p^{-1}(W) \subseteq N_p(\G \ast p)$ is a linear subspace for some $p \in M$ with $\pi(p) = x$. If $C \subseteq M^*$ is locally convex and $x \in C$ then $x$ is called a linear point of $C$ if $T_xC$ is linear.
\end{definition}
\begin{cor}\label{submanifold}
 Let $C \subseteq M^*$ be locally closed and locally convex. Then $\pi^{-1}(C) \subseteq M$ is a smooth submanifold of $M$ without boundary if and only if all points $x \in C$ are linear.
\end{cor}
\begin{proof}
The only if part is easy, so we only prove the if part. Let $x \in C$ be arbitrary. It is enough to show that an open neighborhood of $\pi^{-1}(\{x\})$ in $\pi^{-1}(C)$ is a smooth submanifold without boundary of $M$. By lemma \ref{regular point} it follows that $C$ does not contain a regular boundary point. Therefore, by lemma \ref{exponentiating interior points}, for all sufficiently small $\rho > 0$ we have $$\exp_x(B_\rho(0_x) \cap T_xC) = B_\rho(x) \cap C.$$ Let $p \in M$ with $\pi(p) = x$ and set $V^{< \rho} := d\pi_p^{-1}(B_\rho(0_x) \cap T_xC)$. Then 
$$\pi^{-1}(C) \cap B_\rho(\G \ast p) = \G \ast \exp_p(V^{< \rho}).$$
Choosing $\rho > 0$ sufficiently small, and since $d\pi_p^{-1}(T_xC)$ is a $\G_p$-invariant linear subspace of $N_p(\G \ast p)$, it follows that $\G \ast \exp_p(V^{< \rho})$ is a smooth submanifold of $M$ without boundary. 
\end{proof}
\begin{definition}
 A geodesic $c$ of $M$ is called horizontal if $\dot c(t)$ is perpendicular to $\G \ast c(t)$ for all $t$. A curve $\gamma$ in $M^*$ is called a horizontal geodesic if $\gamma = \pi \circ c$, where $c$ is a horizontal geodesic of $M$.
\end{definition}
A geodesic of $M^*$ is also a horizontal geodesic but a horizontal geodesic in general is a broken geodesic. Observe that a horizontal geodesic of $M^*$ defined on some nonempty open interval can always be uniquely extended to a horizontal geodesic of $M^*$ defined on $\mathbb R$ given that $M$ is complete. It is convenient to make the following definition which extends the standard terminology.
\begin{definition}
Let $x \in M^*$ and $v \in T_xM^*$. Then 
$$\exp_x(tv) := \pi(\exp_p(t\hat v)),$$
where $p \in M$ with $\pi(p) = x$ and $\hat v \in N_p(\G \ast p)$ with $d\pi_p(\hat v) = v$.
\end{definition}
Consequently $\exp_x$ is defined on all of $T_xM^*$ if $M$ is complete and the map $t \mapsto \exp_x(tv)$ is a horizontal geodesic for all $v \in T_xM^*$. 
\begin{lemma}\label{extension}
 Let $C \subseteq M^*$ be locally closed and locally convex and $\gamma : [0,l] \to C$ be a horizontal geodesic such that $\gamma(l)$ is not contained in the closure of the set of nonlinear boundary points of $C$. Then there exists $\epsilon > 0$ and an extension of $\gamma$ to a horizontal geodesic $\gamma : [0,l + \epsilon] \to C$. In particular if $C$ is closed and does not contain any nonlinear boundary point we have $C = \exp_x(T_xC)$ for all $x \in C$.
\end{lemma}
\begin{proof}
If $\dim C = 0$, i.e. $C$ consists of a single point, the claim is trivial. Otherwise we argue by induction on $n = \dim C \geq 1$: First assume that $\dim C = 1$. 
Then $C$ is either isometric to a circle, the real line, $]0,\infty[$, $[0,\infty[$ or a possibly noncomplete bounded interval. If $C$ is a circle or the real line it has empty boundary and every horizontal geodesic contained in $C$ is in fact a geodesic and can be extended infinitely within $C$ so the claim holds. If $C = [a,b]$ is a closed interval it is clear that every horizontal geodesic $\gamma$ in $C$ can be extended within $C$ until it hits the boundary of $C$. Therefore, we have to show that a geodesic $\gamma : [0,l] \to C$ with $\gamma(l) \in \partial C$ can be extended to a  horizontal geodesic $\gamma : [0,l + \epsilon] \to C$ for some $\epsilon > 0$, given that the boundary point $\gamma(l)$ is a linear point of $C$. This is not difficult using the fact that $\G_p$ acts with cohomogeneity one on the linear space $d\pi_p^{-1}(T_{\gamma(l)}C)$, where $\pi(p) = \gamma(l)$. The details are left to the reader. The remaining cases follow by combination of the arguments of this two cases.

Now let $\dim C \geq 2$ and the claim hold for smaller dimensions. Denote by $E$ the closure of the set of nonlinear boundary points of $C$. 
Let $\gamma : [0,l] \to C$ be a horizontal arc length geodesic such that $\gamma(l) \in C \setminus E$. Let $\gamma_-(t) = \gamma(l - t)$ and $v = \dot \gamma_- (0) \in \Sigma_{\gamma(l)}C$. Since $\dim C \geq 2$ and $C$ is locally closed, $\Sigma_{\gamma(l)}C$ is a closed and convex subset of $\Sigma_{\gamma(l)}M^*$ of dimension $\dim C - 1$. We show that $\Sigma_{\gamma(l)}C$ contains no nonlinear boundary point: By lemma \ref{exponentiating interior points} and since $x \in C \setminus E$, there exists $\rho > 0$ such that 
\begin{align}\label{im late}
\exp_x : B_\rho(0_x) \cap T_xC \to B_\rho(x) \cap C
\end{align}
is a homeomorphism and all boundary points of $B_\rho(x) \cap C$ are linear. 
Let $y \in \partial (B_\rho(x) \cap C)$. Because $y$ is linear it follows again by lemma \ref{exponentiating interior points} that a neighborhood $U$ of $y$ in $B_\rho(x) \cap C$ lifts to a smooth submanifold of $M$ under $\pi$. Since \eqref{im late} is a homeomorphism and the exponential map of $M$ is smooth, it follows that all points $v \in \partial (B_\rho(0_x) \cap T_xC)$ are linear. Since $T_xC$ is a cone, the claim follows.

Now let $\tau > 0$ and $\alpha : [0,\tau] \to \Sigma_{\gamma(l)}C$ be an arc length geodesic with $\alpha(0) = v$. Then we may apply the induction hypothesis to find an extension of $\alpha$ to a  horizontal geodesic $\alpha : [0,\infty[ \to \Sigma_{\gamma(l)}C$. Let $p \in M$ and $\hat v \in N_p(\G \ast p)$ with $v = d\pi_p(\hat v)$. Since $\alpha(0) = v$ it follows that $\alpha(\pi) = d\pi_p(-\hat v) =:v^- \in \Sigma_{\gamma(l)}C.$ 
Now, by lemma \ref{exponentiating interior points} there exists an $\epsilon > 0$ such that $\exp_{\gamma(l)}(tv^-) \in C$ for all $0 \leq t \leq \epsilon$. It follows that the horizontal geodesic of $M^*$ given by
$$t \mapsto \exp_{\gamma(0)}(t\dot \gamma(0))$$
maps to $C$ for $t \in [0,l + \epsilon]$ and we are done.
\end{proof}
The set of nonlinear boundary points of $C$ plays an important role for our arguments. For example as we see from corollary \ref{submanifold} this set needs to be empty if we want $\pi^{-1}(C) \subseteq M$ to be a smooth submanifold without boundary. Therefore we make the following definition.
\begin{definition}
Let $C \subseteq M^*$ be locally closed and locally convex. Then we denote by $E$ the closure of the set of nonlinear boundary points of $C$.
\end{definition}

\begin{cor}\label{minus}
 Let $C \subseteq M^*$ be locally closed and locally convex. Let $x \in C \setminus E$ and $p \in M$ with $\pi(p) = x$. Then for every $v \in N_p(\G \ast p)$ with $d\pi_p(v) \in T_xC$  also $d\pi_p(-v) \in T_xC$.
\end{cor}
\begin{proof}
 This follows easily from lemma \ref{extension}.
\end{proof}

From corollary \ref{submanifold} we see that the regularity of $\pi^{-1}(C)$ is determined by the tangent cones of $C$. Another object connected to the local geometry of the embedding $C \subset M^*$ at $x$ is the normal $N_xC$ to $C$ at $x$. 
\begin{definition}
 Let $C \subseteq M^*$ be locally convex and $x \in C$. Then the normal cone to $C$ at $x$ is defined as
 $$N_xC := T_xC^\perp = \{v \in T_xM^* \mid \measuredangle (v,T_xC) \geq \pi/2\} \cup \{0_x\},$$
 where $0_x$ denotes the apex of $T_xM^*$. By convention the normal cone to a one point set $\{x\}$ is given by $T_xM^*$. $N_xC$ is called trivial if it equals $\{0_x\}$.
\end{definition}
Given a convex subset $C$ of a Riemannian manifold and $x \in C$ we have $N_xC^\perp = T_xC$, so the the geometry of the normal cone and the tangent  cone determine each other. For convex subsets of quotient spaces this is no longer true as it can be seen from examples \ref{example 1} and \ref{example 2}. This is quite unfortunate, since it is a general phenomenon that the structure of the normal cone is easier to determine than the one of the tangent cone, as it will appear frequently throughout the rest of the arguments. The main reason for this is the following observation.
\begin{lemma}
Let $C \subseteq M^*$ be locally convex. Let $x \in C$ and $p \in M$ with $\pi(p) = x$. Then $d\pi^{-1}_p(N_xC) \subseteq N_p(\G \ast p)$ is convex.
\end{lemma}
\begin{proof}
Observe that
\begin{align*}
d\pi_p^{-1}(N_xC) &= d\pi_p^{-1}(T_xC^\perp) \\
&= \{v \in N_p(\G \ast p) \mid \sup\{g_p(u,v) \mid {u \in d\pi^{-1}_p(T_xC)} \} \leq 0\}\\
&= \bigcap_{u \in  d\pi^{-1}_p(T_xC)}\{v \in N_p(\G \ast p) \mid g_p(u,v) \leq 0\}.
\end{align*}
Then it is easy to check that the last term of this equation defines a convex set.
\end{proof}
As illustrated by example \ref{example 2} the tangent cone of an interior point of a convex set may fail to be linear while the upcoming corollary shows that the normal cone to an interior point is indeed a linear space.

\begin{cor}\label{linear normal}
 Let $C \subseteq M^*$ be locally closed and locally convex. Let $x \in C \setminus E$. Then $N_xC$ is linear.
\end{cor}
\begin{proof}
Let $x \in C$ and $p \in M$ with $\pi(p) = x$. Set $N := d\pi_p^{-1}(N_xC)$. Then $N$ is a convex cone in $N_p(\G \ast p)$. Since $d\pi^{-1}_p(T_xC)$ is invariant under $\operatorname{-Id}$ by corollary \ref{minus}, the same holds for $N$. But then, since $N$ is convex, it is a linear subspace of $N_p(\G \ast p)$.
\end{proof}

We conclude this section with a useful observation. Let $C \subseteq A$ be a locally convex and locally closed subset of an Alexandrov space. Then we denote by $\partial^\tau C$ the topologicaly boundary of $C$ as a subset of $A$. In particular if $\dim C < \dim A$ we have $\partial^\tau C = C$. The following lemma discusses the situation when $\dim C = \dim A$.
\begin{lemma}\label{boundary}
 Let $A$ be an Alexandrov space and $C \subseteq A$ be locally convex and locally closed with $\dim C = \dim A$. Then
 $$\partial C = \partial^\tau C \cup (C \cap \partial A).$$
\end{lemma}
\begin{proof}
We argue by induction on $n = \dim C = \dim A \geq 1$. The case $n = 1$ is easy and left to the reader. Now assume that $n \geq 2$ and the claim hold for dimensions less than $n$. First let $x \in \partial C$. Then $\Sigma_xC$ is a convex subset of $\Sigma_xA$. 
Hence by the induction hypothesis $\partial \Sigma_xC = \partial^\tau \Sigma_xC \cup (\Sigma_xC \cap \partial \Sigma_xA)$. Since $\partial \Sigma_xC \neq \emptyset$ it follows that $x \in \partial^\tau C$ or $x \in C \cap \partial A$. Now let $x \in C$ with $x \in \partial^\tau C \cup \partial A$. We have to show that $\Sigma_xC$ has nonempty boundary. First let $x \in \partial A$. Assume that $\Sigma_xC$ has empty boundary. Denote by $\hat {\Sigma_xA}$ the double of $\Sigma_xA$ with projection $p : \hat{\Sigma_xA} \to \Sigma_xA$. By induction hypothesis it follows that $\Sigma_xC \cap \partial \Sigma_xA = \emptyset$. Therefore $p^{-1}(\Sigma_xC)$ is a locally convex subset $\hat {\Sigma_xA}$ as well. Since $\hat {\Sigma_xA}$ has empty boundary and so does $p^{-1}(\Sigma_xC)$ it follows from the induction hypothesis that $p^{-1}(\Sigma_xC) = \hat \Sigma_xA$. Therefore $\Sigma_xC = \Sigma_xA$. In particular $\Sigma_xC$ has nonempty boundary, a contradiction. Now let $x \in \partial^\tau C$. We have to show that $x \in \partial C$. For that let $x_n \in A \setminus C$ with $x_n \to x$. Consider a minimal arc length geodesic $\gamma_n : [0,l_n] \to A$ from $C$ to $x_n$. Then it follows that $\dot \gamma(0) \notin \Sigma_{\gamma_n(0)}C$. In particular $\partial^\tau\Sigma_{\gamma_n(0)}C$ is nonempty in $\Sigma_{\gamma_n(0)}A$. By the induction hypothesis $\partial\Sigma_{\gamma_n(0)}C \neq \emptyset$. Hence $\gamma_n(0) \in \partial C$ for all $n$. Clearly $\gamma_n(0) \to x$. Therefore, $x \in \partial C$ since $\partial C$ is closed.
\end{proof}
\subsection{horizontally convex subsets}

\begin{definition}
Let $C \subseteq M^*$. $C$ is called horizontally convex if for every $x,y \in C$ every horizontal geodesic from $x$ to $y$ is contained in $C$. $C$ is locally horizontally convex if $C$ is connected and for all $x \in C$ there exist $\epsilon > 0$ such that for all $y,z \in B_\epsilon(x) \cap C$ every horizontal geodesic from $y$ to $z$ of length less than $\epsilon$ is contained in $C$.
\end{definition}

Horizontally convex subsets are connected to the quotient structure of $M^*$ more closely than convex subsets. Although horizontally convex subsets should be expected to be more rigid than convex subsets it appears to be difficult to obtain general structure results. However, in this section we will obtain the important result that the set of linear points of a horizontally convex subset $C \subseteq M^*$ is open in $C$. For that we first need to discuss a simple notion of conjugacy in $M^*$.

\begin{definition}
 Let $x,y \in M^*$ be regular points and $\gamma : I \to M^*$ be a horizontal geodesic from $x$ to $y$. Then $x$ and $y$ are conjugate along $\gamma$, if there exists a variation $\gamma_s$ of $\gamma$ via horizontal geodesics, such that $\gamma_s$ is smooth at all regular points and the variational field of $\gamma_s$ at $s = 0$ is nontrivial and vanishes at $x$ and $y$.
\end{definition}

 This definition can be extended to arbitrary points of $M^*$. However, this way the concept is easy to handle and it suffices for our needs.

\begin{lemma}\label{conjugate}
 Let $\gamma : [0,\infty[ \to M^*$ be a horizontal geodesic such that $\gamma(0)$ is regular. Then the set $\{t \in [0,\infty[ ~\mid \gamma(t) \text{ is conjugate to  } \gamma(0) \text { along } \gamma_{\vert [0,t]}\}$ is finite when intersected with any compact set. If $\gamma(0)$ and $\gamma(t)$ are not conjugate along $\gamma_{[0,t]}$ there exists an open neighborhood $U$ of $t\dot \gamma (0)$ in $T_{\gamma(0)}M^*$ such that $\exp_{\gamma(0)} : U \to M^*$ is a diffeomorphism onto its image.
\end{lemma}
\begin{proof}
 The first statement follows easily, since the regular part of $M^*$ is a Riemannian manifold. Now let $\gamma : [0,1] \to M^*$ be a horizontal geodesic such that $\gamma(0)$ and $\gamma(1)$ are regular points of $M^*$. Let $\hat \gamma$ be a horizontal lift of $\gamma$. Since $\gamma(0)$ is regular, we may identify $T_{\gamma(0)}M^* = N_{\hat \gamma(0)}(\G \ast \hat \gamma (0))$ and then we have $\exp_{\gamma(0)} = \pi \circ \exp_{\hat \gamma(0)}$. In particular $\exp_{\gamma(0)}$ is smooth in a neighborhood of $\dot \gamma(0)$ and if $d(\exp_{\gamma(0)})_{\dot \gamma(0)}$ has nontrivial kernel there exists a nontrivial Jacobifield $J$ along $\hat \gamma$ with $J(1) \in T_{\hat \gamma(1)}(\G \ast \hat \gamma(1))$ and $J'(0) \in N_{\hat \gamma(0)}\G \ast \hat \gamma(0)$. But then it follows that $\gamma(0)$ and $\gamma(1)$ are conjugate along $\gamma$.
\end{proof}

\begin{lemma}\label{conjugate sequence}
 Let $x \in M^*$ and $\gamma : [0,1] \to T_xM^*$ be a horizontal geodesic whose endpoints are not conjugate along $\gamma$. Let $\gamma_n : [0,1] \to nM^*$ be horizontal geodesics for $n \in \NN$ such that $\gamma_n \to \gamma$ for $n \to \infty$ with respect to the pointed Gromov-Hausdorff convergence $(nM^*,x) \to (T_xM^*,0_x)$. Then for all $\epsilon > 0$ there exists $\delta > 0$ and $N \in \NN$ such that $B^{nM^*}_\delta(\gamma_n(1)) \subset \exp^{nM^*}_{\gamma_n(0)}(B_\epsilon(\dot \gamma_n(0)))$ for all $n \geq N$.
\end{lemma}
\begin{proof}
Let $v = \gamma(0) \in T_xM^*$. Since the endpoints of $\gamma$ are not conjugate along $\gamma$, there exists a neighborhood $U$ of $\dot \gamma(0)$ in $T_vT_xM^*$ such that $\exp_{v} : U \to T_xM^*$ is a diffeomorphism onto its image. 
 In particular for all $\epsilon > 0$ there exits $\delta > 0$ such that $B_\delta(\gamma(1)) \subset \exp^{T_xM^*}_{\gamma(0)}(B_\epsilon (\dot \gamma(0)))$. Then via the Gromov-Hausdorff approximations $\exp^{nM^*}_x : T_xM^* \to nM^*$ this property carries over to the approximating sequence. We leave the details to the reader.
\end{proof}
The next lemma will be of technical help at various occasions.
\begin{lemma}\label{Fix}
Let $q \in M$ and denote by $N$ the component of $\operatorname{Fix}(\G_q)$ containing $q$. Let $N(\G_q)$ denote the normalizer of $\G_q$ and $\mathsf H$ denote the subgroup of elements of $N(\G_q)$ that leave $N$ invariant. Then $\mathsf H$ is acting isometrically on $(N,g)$ and a horizontal geodesic of $(N,g)$ with respect to the action of $\mathsf H$ is also a horizontal geodesic of $(M,g)$ with respect to the action of $\G$. Moreover the map 
\begin{align*}
 h : N/\mathsf H &\to M^*\\
 pr(p) &\mapsto \pi(p)
\end{align*}
is an isometric embedding onto its image $\pi(N)$ equipped with the induced intrinsic metric, where $pr : N \to N/\mathsf H$ denotes the projection.
\end{lemma}
\begin{proof}
Let $U$ denote the subgroup of elements of $\G$ that leave $\operatorname{Fix} \G_q$ invariant. Then $U = N(\G_q)$: To see this first let $u \in U$. Then $u\G_qu^{-1} = \G_{uq} \supseteq \G_q$, since $uq \in \operatorname{Fix}(\G_q)$ for $u \in U$. Since $\G_q$ is compact it follows that $\G_q = u\G_qu^{-1}$. So $U \subseteq N(\G_q)$. On the other hand let $n \in N(\G_q)$ and $p \in \Fix (\G_q)$. Then $\G_{np} = n\G_pn^{-1} \supseteq n\G_qn^{-1} = \G_q$, so $np \in \Fix (\G_q)$ and $N(\G_q) \subseteq U$. 

Observe that by the slice theorem the set of points of $N$ of the same type as $q$ form an open and dense subset of $N$. Also for $p \in \Fix (\G_q)$ we have $\G_p \supseteq \G_q$ and therefore if $p$ and $q$ have the same type we have $\G_p = \G_q$. Next we show that 
\begin{align}\label{groups}
(\G \ast p) \cap \operatorname{Fix}(\G_q) = N(\G_q) \ast p
\end{align}
for all $p \in \operatorname{Fix} \G_q$ of the same type as $q$: Since $U = N(\G_q)$ it follows that $N(\G_q) \ast p \subseteq \G \ast p \cap \Fix(\G_q)$. On the other hand let $g \in \G$ such that $gp \in \Fix (\G_q)$. Then $g\G_qg^{-1} = g\G_pg^{-1} = \G_{gp} \supseteq \G_q$. Again it follows that $g\G_qg^{-1} = \G_q$, so $g \in N(\G_q)$ and \eqref{groups} is proven. From \eqref{groups} it moreover follows that
\begin{align}\label{groups2}
 (\G \ast p) \cap N = \mathsf H \ast p
\end{align}
for all $p \in N$ of the same type as $q$. Now let $\gamma : [0,1] \to N$ be a horizontal geodesic with respect to the action of $\mathsf H$. We show that $\gamma : [0,1] \to M$ is a horizontal geodesic with respect to the $\G$-action as well. Let $p = \gamma(0)$. Since $N$ is totally geodesic, $\gamma$ is a geodesic of $M$ and therefore it suffices to show that $\dot \gamma(0) \in N_p(\G \ast p)$. Since the set of points with istropy group $\G_q$ is open and dense in $N$ we may assume that $\G_p = \G_q$. 
Consider the decomposition
$$T_pM = N_p(\G \ast p) \oplus T_p(\G \ast p) = F^N \oplus P^N \oplus F^T \oplus P^T,$$
where $F^N$ and $F^T$ denote the respective fixed point sets of the $\G_q$-actions on $N_p(\G \ast p)$ and $T_p(\G \ast p)$ and $P^N$ and $P^T$ denote their respective orthogonal complements. Then $T_pN = F^N \oplus F^T$. From \eqref{groups2} it follows that $T_p(\mathsf H \ast p) = T_p(\G \ast p) \cap T_pN = F^T$.  Therefore $F^N$ equals the normal space to $\mathsf H \ast p$ at $p$ in $N$ and $\dot \gamma(0) \in F^N \subset N_p(\G \ast p)$.

Finally we consider the map
\begin{align*}
h : N/\mathsf H &\to \pi(N).\\
pr(p) &\mapsto \pi(p)
\end{align*}
Since $\mathsf H$ is a subgroup of $\G$, $h$ is well defined. Denote by $N_0$ the set of points of $N$ of type $\G_q$. It follows from \eqref{groups2} that $h : N_0 /\mathsf H \to \pi(N_0)$ is a homeomorphism. Moreover, since a horizontal geodesic of $N$ with respect to the action of $\mathsf H$ is also a horizontal geodesic of $M$ with respect to the action of $\G$, it follows that $h : N_0/\mathsf H \to \pi(N_0)$ is an isometry, when $\pi(N_0)$ is equipped with the induced intrinsic metric. Hence $h : N/\mathsf H \to \pi(N)$ is an isometry as well by continuity.
\end{proof}
\begin{prop}\label{linear is open}
 Let $C \subset M^*$ be locally closed and locally horizontally convex. Then the set of linear points is open in $C$.
\end{prop}
Before proving this we give an example illustrating that the corresponding statement does not hold for convex subsets of $M^*$. 

\begin{example}\label{disc}
Let $\mathbb Z_2$ act on $\mathbb R^2$ by reflection along the $x$-axis. Then $\mathbb R^2 /\mathbb Z_2$ is isometric to a halfplane $H$. Let $C \subset H$ be a closed disc tangent to the boundary of $H$. Then the set of linear points of $C$ consists of the interior points of $C$ together with the boundary point of $C$ tangent to the boundary of $H$. 
\end{example}
We suggest to keep this example in mind while going through the proof.
\\ 

\noindent \textit{Proof of proposition \ref{linear is open}.} We proof the claim by induction on $n = \dim M^* \geq 1$. If  $n = 1$, then $C$ is either a single point, and the claim is trivial, or $\dim C = \dim M^* = 1$ and the claim follows easily. We leave the details to the reader.

Now let $\dim M^* \geq 2$ and the claim be proven for smaller dimensions. We may assume that $\dim C \geq 1$, again since the statement is trivial otherwise. 

First we discuss the case that $C$ does not contain a point with principal isotropy: Let $x = \pi(p) \in C$ be a regular point in $C$ of maximal type $\G_p$. Let $N$ denote the fixed point component of the $\G_p$-action on $M$ containing $p$. Then $C \subseteq \pi(N)$. From lemma \ref{Fix} it follows that $\pi(N)$, equipped with the induced intrinsic metric, is isometric to $N/\mathsf H$, where $\mathsf H$ denotes the subgroup of $N(\G_p)$ that leaves $N$ invariant. Moreover $C$ considered as a subset of $N/\mathsf H$ is horizontally convex, again by lemma \ref{Fix}. Since $p$ has nonprincipal type, $\dim N/\mathsf H = \dim \pi(N) < \dim M^*$. Hence the claim follows from the induction hypothesis. 

Therefore we may assume that regular points of $C$ are also regular points of $M^*$. Set $V := d\pi_p^{-1}(T_xC)$. Since $V$ is a linear space invariant under $\G_p$, there exists $\epsilon > 0$ such that 
$$U := \G \ast (\exp_p(B_\epsilon(0_p) \cap V))$$
is a smooth submanifold of $M$ without boundary. Clearly $\G$ acts by isometries on $(U,g)$. Denote by $U^*$ the quotient space $(U,g)/\G$. Considering $U^*$ as a subset of $M^*$ the quotient metric on $U^*$ coincides with the intrinsic metric induced from $M^*$. From the construction it is clear that $T_xU^* = T_xC = V/\G_p$. It then follows from local convexity of $C$ that $(B_\epsilon(x) \cap C) \subseteq U^*$. Observe that $B_\epsilon(x) \cap C$ is again locally horizontally convex. Therefore we may for simplicity assume that $C \subseteq U^*$. Since distances of points in $U^*$ are bigger when measured in $U^*$ than in $M^*$ and $C$ is locally convex in $M^*$, it moreover follows that $C \subseteq U^*$ is locally convex as well. 
To prove the claim it is now clearly enough to show that $x$ is an interior point of $C$ considered as a subset of $U^*$: 

Case 1: Let $x \notin \partial U^*$: Then $\partial T_xU^* = \emptyset$. Since $T_xC = T_xU^*$ and $C$ is locally convex in $U^*$ it follows from lemma \ref{boundary} that $x \notin \partial^\tau C$, the topological boundary of $C$ in $U^*$. So $x$ is an interior point of $C$ in $U^*$ and we are done.

Case 2. Let $x \in \partial U^*$: Let $F$ be a boundary face of $U^*$ containing $x$, i.e. a type component of $U^*$ of codimension $1$ and whose closure contains $x$. In a first step we show that a neighborhood of $x$ in $F$ is contained in $C \cap F$. For that we aim at the induction hypothesis. Since $\pi^{-1}(F)$ in general is not a submanifold of $M$ the arguments for this are a bit complicated.

By the slice theorem there exists $q \in U$ with $\pi(q) \in F$, $\G_q \subseteq \G_p$ and the set of points of $F$ of type $\G_q$ is open and dense in $F$. Denote by $N \subset M$ the component of $\operatorname{Fix}(\G_q)$ containing $q$. Then $p \in N$. Let $\mathsf H$ denote the subgroup of $N(\G_q)$ leaving $N$ invariant. By lemma \ref{Fix} $\mathsf H$ is acting isometrically on $N$. Let $(N,g)/\mathsf H =: N^*$. Then the map
\begin{align*}
h :~ N^* &\to M^*\\
pr(u) &\mapsto \pi(u)
\end{align*}
is an isometric embedding onto its image $\pi(N)$, again by lemma \ref{Fix}. Observe that $h(N^*) \cap U^* = F$ and hence $h(N^*) \cap C = C \cap F$. Therefore $h^{-1}(C \cap F)$ is locally horizontally convex in $N^*$, since $h$ maps horizontal geodesics of $N^*$ to horizontal geodesics of $M^*$, once more by lemma \ref{Fix}. 

Now let $z \in N^*$ with $h(z) = x$. We show that $$T_z(h^{-1}(C \cap F)) = T_zh^{-1}(F):$$ Since $h$ is isometric, it is enough to show that $T_x(C \cap F) = T_xF$. Clearly $T_x(C \cap F) \subset T_xF$. So let $v \in T_xF$ and we have to show that $v \in T_x(C \cap F)$. Without loss of generality we may assume that $v$ is an interior point of $T_xF$ considered as a subset of $\partial T_xU^*$. Since $C \subset U^*$ and $C$ contains a regular point of $M^*$, a regular point of $U^*$ is also a regular point of $M^*$. Since $T_xF \subset T_xU^*$ and $T_pU$ is linear there exists a horizontal geodesic $\gamma : [0,1] \to T_xU^*$  with $\gamma(1/2) = v$ such that $\gamma(0)$ and $\gamma(1)$ are regular points of $T_xM^*$. Using lemma \ref{conjugate} we may further assume that $\gamma(0)$ and $\gamma(1)$ are not conjugate along $\gamma$ in $T_xM^*$. 
Now consider the pointed Gromov-Hausdorff limit
$$(T_xM^*,T_xU^*,0_x) = \lim_{n \to \infty}(nM^*,nU^*,x).$$ Using lemma \ref{conjugate sequence} we can realize $\gamma$ as the limit of a sequence of horizontal geodesics $\gamma_n : [0,1] \to nM^*$ such that $\gamma_n(0) \in nU^*$ and $\gamma_n(1) \in nU^*$ for all $n \in \NN$. Since the points $\gamma(0)$ and $\gamma(1)$ are regular it follows that $\gamma(0), \gamma(1) \notin \partial T_xU^*$. Since $(nC,x) \to (T_xC,0_x) = (T_xU^*,0_x)$ it follows that $\gamma_n(0) \in nC$ and $\gamma_n(1) \in nC$ for all sufficiently large $n$ (otherwise there are points of $\partial (nC)$ arbitrary close to $\gamma_n(0)$ or $\gamma_n(1)$ for all sufficiently large $n$ and it follows that $\gamma(0) \in \partial T_xU^*$ or $\gamma(1) \in \partial T_xU^*$). Therefore, since $C$ is locally horizontally convex, we conclude $\gamma_n([0,1]) \subset nC \subset nU^*$ for all $n$ sufficiently large. Since $\gamma(1/2) = v \in T_xF \subset \partial T_xU^*$ there exists a sequence $t_n$ with $t_n \to 1/2$ for $n \to \infty$ and $\gamma_n(t_n) \in \partial nU^* \cap nC$ for all sufficiently large $n$. Since $v$ is an interior point of $T_xF$ in $\partial T_xU^*$ it moreover follows that $\gamma_n(t_n) \in nF \cap nC$ for all suffieciently large $n$ and we conclude $v \in T_x(C \cap F)$. 

Now observe that $h^{-1}(F) = pr(N \cap U)$ and $N \cap U$ is an $\mathsf H$-invariant smooth submanifold of $N$, since it is the fixed point set of the induced $\G_q$-action on $U$. Since $T_z(h^{-1}(C \cap F)) = T_zF$, it follows from the induction hypothesis combined with lemma \ref{submanifold} that a neighborhood of $z$ in $h^{-1}(C \cap F)$ is also a neihgborhood of $z$ in $h^{-1}(F)$. Consequently a neighborhood of $x$ in $C \cap F$ is also a neighborhood of $x$ in $F$ and we are done with the first step.

Now let $\hat U^*$ denote the double of $U^*$ obtained by gluing together two copies of $U^*$ along their common boundary. Also let $\hat C \subset \hat U^*$ denote the set of points that map to $C$ under the canonical map $\hat U^* \to U^*$. Since the first step applies to every boundary face of $U^*$ containing $x$, it follows that a small open neighborhood of $x$ in $\partial U^*$ is contained in $C$. Therefore a small open neighborhood of $x$ in $\hat C$ is locally convex in $\hat U^*$. Since $T_xC = T_xU^*$ it follows that $T_x\hat C = T_x \hat U$. Thus again it follows from lemma \ref{boundary} that $x$ is an interior point of $\hat C$ in $\hat U^*$. But then $x$ is an interior point of $C$ in $U^*$ as well and the proof is finished. \hfill$\square$


\subsection{Extemal subsets and boundary strata}\label{extremal}
In this short section we recall the definition of extremal subsets of an Alexandrov space and some of its properties which are important to us. If not mentioned otherwise we refer to the discussion given in the survey \cite{petrunin2007} and the refrences therein. 

\begin{definition}
A closed subset $E$ of an Alexandrov space $A$ is called extremal if the following holds: For every $q \in A\setminus E$ and every $p \in E$ such that the distance function $d_q$ to $q$ restricted to $E$ has a local minimum at $p$ the point $p$ is critical for $q$. 
\end{definition}
Note that $A$ as well as the empty set are extremal. Also the intersection of two extermal subsets is again extermal. Therefore for $x \in A$ there exists a smallest extremal subset of $A$ which contains $x$ and which is denoted $Ext(x)$. An extremal subset $E$ is called primitive if $E = Ext(x)$ for some $x \in A$. The main part of a primitive extremal set $E = Ext(x)$ is defined as the set $\{y \in E \mid Ext(y) = Ext(x)\}$. Each main part of a primitive extremal subset $E$ is a topological manifold which is open and dense in $E$ and therefore a natural notion of the dimension of a primitive extremal subset is induced. $A$ is stratified into topological manifolds by the main parts of its primitive extremal subsets and the boundary $\partial A$ of $A$ is obtained as the union of all primitive extremal subsets of codimension $1$.
\begin{definition}
A boundary stratum of an Alexandrov space $A$ is any union of primitive extremal subsets of codimension $1$.
\end{definition}
In the unpublished preprint \cite{perelman91} it was shown that the distance function to the boundary of a nonnegatively curved Alexandrov space $A$ is concave. It is well known that this holds more generally for any boundary stratum of $A$. A detailed proof can be found in \cite{worner2010}.

\begin{lemma}\label{distance boundary}
Let $A$ be a nonnegatively curved Alexandrov space with nonempty boundary $\partial A$. Then the distance function $d_B$ to any boundary stratum $B$ of $A$ is concave on $A \setminus B$.
\end{lemma}
\subsection{The distance function to a boundary stratum and its level sets.}\label{subsection distance} For this section we assume additionally that $M^*$ is compact and has nonnegative curvature. Also we fix a closed and locally horizontally convex subset $\Omega^* \subseteq M^*$ with nonempty boundary and set $\Omega = \pi^{-1}(\Omega^*)$. We fix a boundary stratum $B$ of $\Omega$ and denote by $f = d_B : \Omega \to \RR$ the distance function to $B$. Moreover we assume that $\Omega \setminus \pi^{-1}(B)$ is a smooth submanifold of $M$ without boundary. Since $M^*$ has nonnegative curvature the same holds for $\Omega^*$. Therefore we have the following lemma which is the basis for the arguments in this section.
\begin{lemma}
 $f : \Omega^* \setminus B \to \RR$ is concave.
\end{lemma}
For $l \in \RR$ we set
$$C^l := \{x \in \Omega^* \mid f(x) \geq l\}.$$
Since $f$ is concave, all these sets are convex in $\Omega^*$. An important observation is that they are in fact horizontally convex.
\begin{lemma}\label{horizontally convex}
$C^l$ is horizontally convex in $\Omega^* \setminus B$ for every $l \in \RR$.
\end{lemma}
\begin{proof}
Let $\gamma : [0,1] \to \Omega \setminus \pi^{-1}(B)$ be a horizontal geodesic. Let $0 \leq t_1 < \dots < t_k \leq 1$ such that $\gamma_{\vert [0,1] \setminus \{t_1, \dots t_k\}}$ has constant type. Denote by $\hat f : \Omega \to \RR$ the distance function to $\pi^{-1}(B)$ and let
$$(\hat f \circ \gamma)^-(t) = \lim_{s \searrow 0} \frac{\hat f \circ \gamma(t) - \hat f \circ \gamma(t - s)} t$$
and
$$(\hat f \circ \gamma)^+(t) = \lim_{s \searrow 0} \frac{\hat f \circ \gamma(t) - \hat f \circ \gamma(t + s)} t.$$
It is enough to show that $(\hat f \circ \gamma)^+(t_k) \leq -(\hat f \circ \gamma)^-(t_k)$ for all $k \in \{1, \dots ,n\}$ with $0 < t_k <1$, since $\hat f \circ \gamma$ is concave on every connected interval contained in $[0,1] \setminus \{t_1, \dots t_k\}$. Let $\Gamma_k$ denote the set of initial directions of minimal geodesics from $\gamma(t_k)$ to $\pi^{-1}(B)$. Then 
\begin{align*}
 -(\hat f \circ \gamma)^-(t) = \cos(\measuredangle(-\dot \gamma(t_k),\Gamma_k)) &\geq \cos(\pi - \measuredangle (\dot \gamma(t_k),\Gamma_k))\\
 &= -\cos(\measuredangle(\dot \gamma(t_k),\Gamma_k) = (\hat f \circ \gamma)^+(t).
\end{align*}
\end{proof}
In the following let $a$ denote the maximal value of $f$ and we will only consider the set $C^a \subset \Omega^*$ of points of maximal distance to $B$. But we remark that the corresponding results hold for all the other level sets $C^l$ with different but much simpler proofs. 
The following proposition is the key technical observation.
\begin{prop}\label{tangent linear iff normal linear}
$T_xC^{a}$ is linear if and only if $N_xC^{a}$ is linear. 
\end{prop}
We note that the corresponding proposition does not hold for convex subsets as illustrated by the examples \ref{example 1} and \ref{example 2}. We also note two questions which we were not able to answer in order to motivate the arguments of the proof. Consider a horizontally convex subset $C \subseteq M^*$. Then the first question is whether for $x \in C$ also $T_xC \subseteq T_xM^*$ is horizontally convex and secondly if for every $x \in C$ with $T_xC \neq T_xM^*$ the normal space $N_xC$ is nontrivial. In this case the proof follows easily: If $N_xC$ is linear then $N_xC^\perp = V/\G_p$ for a linear subspace $V \subset N_p(G \ast p)$ and $T_xC \subseteq V/\G_p$. Assume that $T_xC \subsetneq V/\G_p$. Then $T_xC$ has nontrivial normal space considered as a horizontally convex subset of $V/\G_p$ in contradiction to the fact that $N_xC \cap V/\G_p = \{0_x\}$. Our proof follows this line of thought using additionally the extra structure we obtain from the distance function $f$. It is propably correct that the equation $N_xC^\perp = T_xC$ holds for a horizontally convex subset $C$ for every $x \in C$, which would imply that the answer to both questions is generally yes.
\begin{proof}
For simplicity we assume that $\Omega^* = M^*$. The general case follows with identical arguments restricted to $\Omega^*$. Also let $C := C^{a}$ for simplicity of notation. Moreover we may assume that $\dim C \geq 1$, since otherwise the claim is trivial. We first give a proof assuming that $C$ contains a regular point of $M^*$ and afterwards we reduce the general case to this situation. 

Let $x \in C$. Clearly $N_xC$ is linear if $T_xC$ is linear. Therefore we assume that $N_xC$ is linear and have to show that $T_xC$ is linear. If $x$ itself is a regular point of $C$ and therefore also of $M^*$ it follows that $T_xC = N_xC^\perp$ and we are done. Therefore we assume that $x$ is nonregular. 

We use a sequence of convex subsets of $T_xM^*$ that converge to $T_xC$, constructed via limits of the rescaled super level sets of $f$. This will help to determine the geometry of $T_xC$.
\begin{sublem}\label{convex family}
 There exists a family $\{C_t\}_{t \in [0,1]}$ of closed convex subsets of $B_1(0) \subset T_xM^*$ satisfying the following
 \begin{itemize}
  \item[a)] $C_t \subset C_s$ for $t < s$,
  \item[b)] $d_H(C_t,C_s) \to 0$ for $s \to t$, where $d_H$ denotes Hausdorff distance,
  \item[c)] $C_1 = B_1(0)$,
  \item[d)] $C_0 = T_xC \cap B_1(0)$,
  \item[e)] For all $t \in [0,1]$ and $n \in \NN$ there exists $l_t(n) \in \RR$ such that $$(T_xM^*,C_t,0) = \lim_{n \to \infty}(nM^*,nC^{l_t(n)} \cap B^{nM^*}_1(x),x).$$
 \end{itemize}
 \end{sublem}
\noindent \textit{Proof of sublemma \ref{convex family}.} The desired family is constructed via the various Gromov-Hausdorff limits of the rescaled level sets $nC^l$ for $n \in \NN$ and $l \in \RR$ as indicated by property $e)$. The details of the construction are left to the reader. Alternatively see \cite{spindeler14} or \cite{spindeler15}, lemma 3.23. \hfill $\square$.
\\

Now let $C_t$ be a family as given by sublemma \ref{convex family}. Let $p \in M$ with $\pi(p) = x$. Then by assumption $N := d\pi_p^{-1}(N_xC) \subseteq N_p(\G \ast p)$ is a linear $\G_p$-invariant subspace. Let $F \subseteq N$ denote the fixed point set of the $\G_p$-action on $N$. Then we have an orthogonal decomposition
$$N_p(\G \ast p) = F \oplus W.$$
Let
$$\rho : N_p(\G \ast p) \to W$$
denote the orthogonal projection along $F$. Note that $\rho$ naturally induces a projection
$$\rho^* : T_xM^* \to W/\G_p =: W^*.$$ As always we identify $F$ with $F/\G_p \subset T_xM^*$. Finally set 
$$\overline C_t = \rho^*(C_t).$$
Observe that $\overline C_0 = C_0 = T_xC \cap B_1(0)$ by property d) together with $T_xC \subseteq W^*$. Hence we want to prove that $T_0\overline C_0 = T_xC$ is linear, where $T_0\overline C_0$ denote the tangent cone to $\overline C_0$ at $0 = 0_x$. If $0$ is an interior point of $\overline C_0$ considered as a subset of $W^*$ it follows that $T_0\overline C_0 = W^*$ and we are done, since $W^*$ is linear. Therefore we may assume that $0$ is not an interior point of $\overline C_0$ in $W^*$.
Let
$$t_0 := \inf \{ t \in [0,1] \mid 0 \text{ is an interior point of } \overline C_t \text{ in } W^*\}.$$
From property c) of sublemma \ref{convex family} it follows $0 \leq t_0 < 1$.  Also note that $0$ is an interior point of $\overline C_t$ in $W^*$ for all $ t > t_0$ by property $a)$.

Step 1: We show that $\overline C_t$ is horizontally convex for all $t_0 \leq t \leq 1$: Observe that properties a) and b) hold analogeously for the family $\overline C_t$ and therefore it suffices to prove this for all $t_0 < t \leq 1$. Let $t_0 < t \leq 1$ be fixed and $\gamma : [0,1] \to T_xM^*$ be a horizontal geodesic with $\gamma(0) \in \overline C_t$ and $\gamma(1) \in \overline C_t$. We have to show that $\gamma ([0,1]) \subseteq \overline C_t$.

\begin{sublem}\label{ray} Let $\G$ act by isometries on a linear  euclidean space $U$ and $U^* = U/\G$. Let $C \subset U^*$ be convex such that $0$ is an interior point of $C$ and let $v \in \partial^\tau C$ (recall that $\partial^\tau C$ denotes the topological boundary of $C$ as a subset of $U^*$). Then for all $0 \leq \lambda < 1$ we have $\lambda v \in \mathring C$, the topological interior of $C$, and for all $\lambda > 1$ we have $\lambda v \notin C$. 
\end{sublem}
\begin{proof}
Once more we leave the proof to the reader.
\end{proof}
Using this lemma we can without loss of generality assume that $\gamma(0)$ and $\gamma(1)$ are interior points of $\overline C_t$ in $W^*$, since otherwise we may approximate $\gamma$ by horizontal geodesics $\lambda \gamma$ with $\lambda < 1$. Since $W^*$ is linear and $\gamma(0), \gamma(1) \in W^*$, it is clear that $\gamma([0,1]) \subset W^*$. Since $C$ contains a regular point of $M^*$ and $T_xC \subset W^*$, a regular point of $W^*$ is also a regular point of $T_xM^*$. Hence way can also  without loss of generality assume that $\gamma(0)$ and $\gamma(1)$ are regular points of $T_xM^*$. Therefore we can further without loss of generality assume that $\gamma(0)$ and $\gamma(1)$ are not conjugate along $\gamma$ in $T_xM^*$ by lemma \ref{conjugate}. 

Now let $v,w \in C_t$ with $\rho^*(v) = \gamma(0)$ and $\rho^*(w) = \gamma(1)$ and choose a horizontal geodesic $\hat \gamma : [0,1] \to T_xM^*$ with $\hat \gamma(0) = v$, $\hat \gamma(1) = w$ and $\rho^* \circ \hat \gamma = \gamma$ (the existence of $\hat \gamma$ follows, since $T_xM^* = F \times W^*$ up to isometry). To show that $\gamma \subset \overline C_t$ we show that $\hat \gamma \subset C_t$: First observe that $v$ and $w$ are not conjugate along $\hat \gamma$ as well (this follows again from $T_xM^* = F \times W^*$). Let 
\begin{align}\label{huiuiui}
(T_xM^*, C_t, 0) = \lim_{n \to \infty} (nM^*, nC^{l_t(n)} \cap B^{nM^*}_1(x),x)
\end{align}
 according to property e) of sublemma \ref{convex family}. Let $\hat \gamma_n : [0,1] \to nM^*$ be a horizontal geodesic for every $n$ such that $\hat \gamma_n \to \hat \gamma$. Since $\hat \gamma(0), \hat \gamma(1) \in C_t$ and by \eqref{huiuiui}, we may by lemma \ref{conjugate sequence} further choose $\hat \gamma_n$ such that $\hat \gamma_n(0) \in nC^{a_t(n)}$ and $\hat \gamma_n(1) \in nC^{a_t(n)}$. Hence by horizontal convexity we deduce that $\hat \gamma_n([0,1]) \subseteq nC^{l(n)}$ for all $n$ and therefore $\hat \gamma ([0,1]) \subseteq C_t$. This completes step 1.

Step 2: We show that $N_0 \overline C_{t_0} \cap W^*$, the normal space to $\overline C_{t_0}$ in $W^*$ at $0$, is nontrivial. From the definition of $t_0$ it is clear that $\overline C_t$ has nonempty topological boundary in $W^*$, which we denote by $\partial^\tau \overline C_t$,  for all $t_0 < t$ sufficiently close to $t_0$. For all such $t$ let $c_t : [0,l_t] \to W^*$ be a minimal  arc length geodesic from $0$ to $\partial^\tau \overline C_t$. Let $t_n \to t_0$ be a sequence with $t_n > t_0$ such that $\dot c_{t_n}(0)$ converges to a limit $v_0$. We claim that $v_0 \in N_0\overline C_{t_0} \cap W^*$: Clearly $v \in W^*$. Assume on the contrary that $v \notin N_0\overline C_{t_0}$. Then, since $\dim T_xC \geq 1$, there exists $b \in \overline C_{t_0}\setminus \{0\}$ such that $2b \in \overline C_{t_0}$ and $\alpha_0 := \measuredangle (b,v_0) < \pi/2$. Let $\alpha_n := \measuredangle (\dot c_{t_n}(0),b)$. Then $\alpha_n \to \alpha_0$ for $n \to \infty$. Let $N \in \NN$ and $\delta > 0$ such that $\alpha_n < \pi/2 - \delta$ for all $n \geq N$. Also let $a_n : [0,1] \to W^*$ be a minimal segment from $b$ to $c_{t_n}(l_{t_n})$. By definition of $t_0$ and $c_{t_n}$ it follows that $c_{t_n}(l_{t_n}) \to 0$ for $n \to \infty$. Therefore there exists $n \in \NN$ such that $\beta_n > \pi/2$ for all $n \geq N$, where $\beta_n$ denotes the angle formed by $c_{t_n}$ and $a_n$ at $c_{t_n}(l_{t_n})$. Since $W^*$ is linear, we may extend every $a_n$ to a horizontal geodesic $a_n : [0,\infty[ \to W^*$. Then $\tilde \beta_n < \pi/2$ for all $N \geq N$, where $\tilde \beta_n$ denotes the angle between $c_{t_n}$ and ${a_n}_{\vert [1,\infty[}$. Fix $n \geq N$. There exists $\epsilon > 0$, such that $a_n(1 + \epsilon)$ is an interior point of $\overline C_{t_0}$ in $W^*$, since $c_{t_n}$ is minimal from $0$ to $\partial^\tau \overline C_{t_n}$. By horizontal convexity of $\overline C_{t_n}$, and since $2b \in \overline C_{t_n}$, it therefore follows that there exists $\lambda > 0$ such that $(1 + \lambda)a_n(t) \in \overline C_{t_n}$ for all $t \in [0,1 + \epsilon]$ (for that note that $(1 + \lambda)\alpha$ is a horizontal geodesic as well). But then also $(1 + \lambda)a_n(1) = (1 + \lambda)c_{t_n}(1) \in \overline C_{t_n}$, in contradiciton to sublemma \ref{ray}.

Step 3: We show that $N_0\overline C_{t_0} \cap W^*$ is linear. Let $N_1 = d\pi_p^{-1}(N_0\overline C_{t_0} \cap W^*)$. Then $N_1$ is a convex subset of $N_p(\G \ast p)$ and $N_1$ is invariant under $\G_p$. Since $C_0 = T_xC \cap B_1(0) \subseteq \overline C_{t_0}$ it follows that $N_1 \subset W \cap N$. Thus, by definition of $W$, the action of $\G_p$ on $N_1 \setminus \{0\}$ is fixed point free. Therefore $N_1$ must be a linear subspace of $N_p(\G \ast p)$  (a $\G_p$-action on a convex cone with nonempty boundary has fixed points different from $0$ coming from the unique direction of maximal angle to the boundary).

To complete the proof we iterate this argument: Define $W_1$ via the orthogonal decomposition 
$$N_p(\G \ast p) = F \oplus N_1 \oplus W_1$$
and set 
$$W_1^* = W_1/\G_p.$$
Note that $\overline C_{t_0} \subset W_1^*$ and therefore $\overline C_t \subset W^*_1$ for all $0 \leq t \leq t_0$.
Moreover, $0$ is an interior point of $\overline C_{t_0}$ in $W^*_1$ (otherwise a normal vector to $\overline C_{t_0}$ at $0$ which is contained in $W_1^*$ can be constructed analogous to step $2$ in contradiction to the definition of $N_1$ and $W_1$).
If $0$ is also an interior point of $\overline C_0$ in $W^*_1$ we are again done. Otherwise set
$$t_1 := \inf \{t \in [0,t_0] \mid 0 \text{ is an interior point of } \overline C_t \text{ in } W_1^*\}.$$
With $W_1^*$ in the role which was played by $W^*$ before the arguments of steps 1 to 3 now apply as well to $\overline C_t$ for $t_1 < t < t_0$ and yield that the normal space to $\overline C_{t_1}$ in $W_1^*$ is nontrivial and linear. Repeating this argument after a finite number of steps we find that $0$ is an interior point of $\overline C_0$ in $W_k^*$ since the dimension of $\overline C_{t_i}$ drops in every step and we are done.

It remains to discuss the case that $C$ does not contain a point with principal isotropy. Let $x \in C$ such that the set of points of $C$ of the same type as $x$ is open and dense in $C$. Let $p \in M$ with $\pi(p) = x$ and denote by $N$ the component of $\operatorname{Fix}(\G_p)$ containing $p$. Then $C \subset \pi(N)$ and $\pi(N)$ equipped with the induced intrinsic metric is isometric to $N/\mathsf H$ by lemma \ref{Fix}, where $\mathsf H \subset N(\G_p)$ denotes the subgroup of elements that leave $N$ invariant. Moreover $C^a \cap \pi(N)$ is horizontally convex in $\pi(N) = N/\mathsf H$ for all the level sets $C^a$, since a horizontal geodesic of $N$ with respect to the acton of $\mathsf H$ is a horizontal geodesic of $M$ with respect to the action of $\G$ as well. Thus, with the same proof as above applied to $N/\mathsf H$ and the family $C^a \cap \pi(N)$ it follows that $T_yC \subset T_y(N/\mathsf H)$ is linear if and only if $N_yC \subset T_y(N/\mathsf H)$ is linear. But clearly $T_yC \subset (N/\mathsf H)$ is linear if and only if $T_yC \subset T_yM^*$ is linear and analogeuously for the normal spaces.
\end{proof}

Recall that for a convex set $C$ we denote by $E$ the closure of the set of nonlinear boundary points of $C$.
\begin{cor}\label{nonlinear boundary closure}
The set of nonlinear points of $C^{a}$ equals the set $E$. In particular all points in $C^a \setminus \partial C^a$ are linear and $\pi^{-1}(C^a \setminus E)$ is a smooth submanifold of $M$ without boundary.
\end{cor}
\begin{proof}
This follows from proposition \ref{tangent linear iff normal linear} in combination with lemma \ref{linear normal} and proposition \ref{linear is open}.
\end{proof}
Now let us construct a soul $\Sigma$ of $M^*$ as follows. Let $\Omega^*_1$ denote the set of maximal distance to $\partial M^*$. If $\partial \Omega^*_1 = \emptyset$ set $\Omega^*_1 =: \Sigma$. In particular it follows from corollary \ref{nonlinear boundary closure} that $\Omega_1 := \pi^{-1}(\Omega^*_1)$ is a smooth closed submanifold. If otherwise $\partial \Omega^*_1 \neq \emptyset$ let $\Omega^*_2$ be the set of points of $\Omega^*_1$ of maximal distance to $\partial \Omega^*_1$. Again if $\partial \Omega^*_2$ is empty set $\Sigma := \Omega^*_2$ and it follows that $\Omega_2 := \pi^{-1}(\Sigma)$ is a smooth closed submanifold. Otherwise consider $\Omega^*_3$ and so on. After a finite number of steps we have constructed a soul $\Sigma$ of $M^*$ such that $\pi^{-1}(\Sigma) \subset M$ is a smooth closed submanifold. 
\begin{cor}
Let $\Sigma \subset M^*$ be a soul constructed as above. Then $\pi^{-1}(\Sigma) \subset M$ is a smooth closed submanifold. 
\end{cor}

Our main theorem \ref{main} does not follow yet, since in general the distance function $d_\Sigma$ may have critical points in $M^* \setminus (B \cup \Sigma)$. In order to prove our main theorem in a similar fashion we thus need to gain control over the regularity of the distance function to a 'soul' (the result $\Sigma \subset M^*$ constructed in the proof of theorem \ref{main} in the next section may have boundary, while $\pi^{-1}(\Sigma)$ does not. Hence the name soul will not be used further). 

\begin{prop}\label{stratum}
$E$ is a boundary stratum of $C^{a}$. Also let 
$$C_1 := \{x \in C^{a} \mid d(x,E) \text{ is maximal }\}.$$
Then $d_{C_1} : \Omega^* \to \RR$ is regular at every $x \in \Omega^* \setminus (B \cup C_1)$.
\end{prop}
\begin{proof}
First we show that $E$ is extremal: Let $x \in C^a \setminus E$ and $y \in E$ such that $d_x$ restricted to $E$ has a local minimum at $y$. We need to show that $y$ is critical for $x$. Assume on the contrary. Let $\Gamma$ denote the set of initial directions of minimal geodesics from $y$ to $x$. Then there exists $v \in \Sigma_yC^a$ with $\measuredangle (v,\Gamma) > \pi/2$. We may choose such a $v$ which is also a regular point of $T_yC^a$ and not contained in $\partial T_yC^a$. Let $\alpha : [0,l] \to \Sigma_yC^a$ be a minimal arc length geodesic from $v$ to $\Gamma$. So $l > \pi/2$. Since $d_x$ restricted to $E$ has a local minimum at $y$ it is clear that $\measuredangle (\alpha(l),T_yE) \geq \pi/2$. Using lemma \ref{extension} it follows that $\alpha$ may be extended horizontally to $\alpha : [0,l + \pi /2] \to \Sigma_yC^a$. Now let $p \in M$ with $\pi(p) = y$ and consider $N := d\pi^{-1}_p(N_yC^a)$. Since $y \in E$ it follows from proposition \ref{tangent linear iff normal linear} that $N$ is convex but not linear. Let $W = N^\perp$. Then $W$ is nonlinear and convex as well and by the splitting theorem we may therefore write 
$$W = \RR^k \times Z$$
up to isometry, where $Z$ is a nontrivial convex cone that does not contain a line. Let $\hat \alpha$ be a horizontal lift of $\alpha$. Since $d\pi^{-1}_p(T_yC^a) \subset W$ it follows that $\hat \alpha : [0,l + \pi/2] \to W$. Since $\hat \alpha$ is parametrized by arc length and $l > \pi/2$ it follows that $\hat \alpha$ in fact maps to $\RR^k \times \{0\}$. Now, since $v = \alpha(0) \in T_yC^a$ is regular and not contained in $\partial T_yC^a$, it follows that $\alpha (0)$ is a linear point of $T_yC^a$ in contradiction to the following observation.

\begin{sublem}\label{W and Z}
Let $y \in C^a$ be a nonlinear point and $W := (d\pi_p^{-1}(N_yC^a))^\perp$. Write up to isometry $W = \RR^k \times Z$, where $Z$ does not contain a line. Then all points in $d\pi_p(\RR^k \times \{0\}) \cap T_yC^a$ are nonlinear points of $T_yC^a$
\end{sublem}
\begin{proof}
Assume that a point $u \in d\pi_p(\RR^k \times \{0\}) \cap T_yC^a$ is a linear point of $C^a$. Let $V := d\pi_p^{-1}(T_yC^a)$ and $\hat u \in V$ with $d\pi_p(\hat u) = u$. Then $T_{\hat u}V$ is linear. Since $V \subseteq W$ and $Z$ does not contain a line it follows that $T_{\hat u}V \subseteq \RR^k \times \{0\}$. But then it follows from convexity of $T_yC^a$ that $V \subseteq \RR^k \times \{0\}$ as well (for that note that $\RR^k \times \{0\}$ is $\G_p$-invariant). But then $Z \subset V^\perp = d\pi^{-1}(N_yC^a)$. This is a contradiction to the definition of $W$, since $Z$ is nontrivial.
\end{proof}

To show that $E$ defines a boundary stratum of $C^a$ it remains to show that $E$ has locally constant codimension $1$ in $C^a$. Let $x \in E$. If $x$ is contained in the closure of the set of regular boundary points it is clear that $x$ is contained in a primitve extremal subset of codimension $1$, since since the set of regular boundary points is open in $\partial C$ and is contained in $E$. Therefore we may assume that $x$ is not contained in the closure of the set of regular boundary points. Then the claim follows from the next sublemma in combination with proposition \ref{tangent linear iff normal linear}.

\begin{sublem}\label{zput}
Let $C \subset M^*$ be locally cosed and locally convex and have no regular boundary point. Further assume that $C$ satisfies the following: For all $x \in C$ the tangent cone $T_xC$ is linear if and only if $N_xC$ is linear. Then, if the set of nonlinear boundary points of $C$ is nonempty, it has codimension $1$ in $C$.
\end{sublem}
\begin{proof}We argue by induction on $n = \dim C \geq 1$. If $n = 1$ the claim follows trivially, since $\partial C$ is discrete in $C$. So let $n \geq 2$ and $x \in E$. Since $C$ does not contain a regular boundary point, by lemma \ref{exponentiating interior points} there exists $\rho > 0$ such that 
\begin{align}\label{homeo}
\exp_x : B_\rho(0_x) \cap T_xC \to C \cap B_\rho(x)
\end{align} 
is a homeomorphism. It follows that $T_xC$ does not contain any regular boundary points as well. We claim that for all $v \in T_xC$ the tangent cone $T_vT_xC$ is linear if and only if the normal cone $N_vT_xC$ is linear:

Clearly $N_vT_xC$ is linear if $T_vT_xC$ is linear. Therefore we assume that $N_vT_xC$ is linear and have to show that $T_vT_xC$ is linear as well. Let $p \in M$ with $\pi(p) = x$, $\hat v \in N_p(\G \ast p)$ with $d\pi({\hat v}) = v$ and $V = d\pi_p^{-1}(T_xC)$. Let $N_{\hat v}V$ denote the normal cone to $V$ at ${\hat v}$, which we consider as an affine subspace of $N_p(\G \ast p)$ centered at $\hat v$. Observe that $N_vT_xC$ is linear if and only if $N_{\hat v}V$ is an affine linear subspace. Let $W = (N_{\hat v}V)^\perp$. Then $W$ is affine linear as well with $T_{\hat v}V \subseteq W$. Let $Conv(T_{\hat v}V)$ denote the convex closure of $T_{\hat v}V$, i.e. the smallest convex subset of $N_p(\G \ast p)$ containing $T_{\hat v}V$. It is straight forward to check that $N_{\hat v}V$ is linear if and only if $Conv(T_{\hat v}V) = W$. Possibly after scaling we may assume that $\rho > |v| = |{\hat v}|$ and that the map $\exp_p : N^{< \rho}_p(\G \ast p) \to M$ is a Diffeomorphism onto its image which is a smooth subma\-nifold of $M$. Now it is again straight forward to check that 
\begin{align}
N_{\exp_p({\hat v})}\pi^{-1}(C) &= (Conv(T_{\exp_p({\hat v})}\pi^{-1}(C)))^\perp\\ &= (Conv((d\exp_p)_{\hat v}(T_{\hat v}V)))^\perp \label{this}\\
&= (d\exp_p)_{\hat v}(Conv(T_{\hat v}V))^\perp \label{that} \\
&= ((d\exp_p)_{\hat v}(W))^\perp.
\end{align}
Here \eqref{this} follows by \eqref{homeo} and \eqref{that} follows, since $d\exp_p$ is linear. Since $W$ is an affine linear space centered at ${\hat v}$, $(d\exp_p)_{\hat v}(W) \subseteq T_{\exp_p({\hat v})}M$ is a linear subspace and therefore $N_{\exp_p({\hat v})}\pi^{-1}(C)$ is linear. Consequently $N_{\exp_x(v)}C$ is linear and therefore  also $T_{\exp_x(v)}C$ is linear by assumption. Then it follows that $T_vT_xC$ is linear as well, again using \eqref{homeo}.

Consequently $\Sigma_xC$ does not contain a regular boundary point and $T_v\Sigma_xC$ is linear if and only if $N_v\Sigma_xC$ is linear for all $v \in \Sigma_xC$. To apply the induction hypothesis we show that the set of nonlinear boundary points of $\Sigma_xC$ is nonempty: Assume on the contrary that it is empty. Then by lemma \ref{extension} every horizontal geodesic of $\Sigma_xC$ can be extended infinitely. Then it follows as in the proof of corollary \ref{minus} that for all $d\pi_p(v) \in T_xC$ we have $d\pi_p(-v) \in T_xC$. Then analogously to corollary \ref{linear normal} it follows that $N_xC$ is linear. But then also $T_xC$ is linear by assumption. Thus, by \eqref{homeo}, all points in a neighborhood of $x$ in $C$ are linear, in contradiction to the choice $x \in E$. Therefore, by the induction hypothesis, the set of nonlinear boundary points has codimension $1$ in $\Sigma_xC$. Hence the same holds for $T_xC$. Thus the same holds as well for $C$, again by $\eqref{homeo}$.
\end{proof}

Finally we turn to the regularity of $d_{C_1}$: Since $d_B : \Omega^* \setminus B \to \RR$ is concave it is easily seen that $d_{C_1}$ is regular at every point $y \in \Omega^* \setminus (B \cup C^{a})$. Also $d_E : C^{a} \to \RR$ is concave, since $E$ is a boundary stratum of $C^a$. Therefore it follows that $d_{C_1}$ is regular at every point $y \in C^{a} \setminus (E \cup C_1)$. It remains to show that $d_{C_1}$ is regular at $x \in E$: Let $\Gamma \subset \Sigma_xC$ denote the set of directions of minimal geodesics from $x$ to $C_1$. Let $N = d\pi_p^{-1}(N_xC)$ and $W = N^\perp = \RR^k \times Z$ as above. From sublemma \ref{W and Z} it follows that
\begin{align}\label{bufa}
d\pi_p^{-1}(\Gamma) \cap (\RR^k \times \{0\}) = \emptyset,
\end{align}
since all points of $\Gamma$ are linear points of $\Sigma_xC$ because the nonlinear points of $C^a$ are given by $E$. Since $Z$ does not contain a line, there exists $u_0 \in d\pi_p^{-1}(N_xC)$ with $\measuredangle (u_0,Z) > \pi /2$. It then follows from \eqref{bufa} that $\measuredangle (d\pi_p^{-1}(\Gamma),u_0) > \pi /2$ as well and we are done.
\end{proof}


\section{proof of theorem \ref{main} and further results}\label{section proof main}
For this section we fix a compact nonnegatively curved Riemannian manifold $M$ equipped with an isometric action by a compact Liegroup $\G$ in way that the quotient space $M^*$ has nonempty boundary. Let $B$ be any boundary stratum of $M$. We begin with the proof of our main theorem.
\\

\textit{Proof of theorem \ref{main}.} We use the results of section \ref{subsection distance} freely. Let $\Omega^*_1 \subset M^*$ denote the set of maximal distance to $B$. Then $\Omega^*_1$ is horizontally convex and the set $E_1$ of nonlinear points of $\Omega^*_1$ defines a boundary stratum of $\Omega^*_1$ if $E_1$ is nonempty.

\textit{Case 1.} Let $E_1 = \emptyset$. Then $N := \pi^{-1}(\Omega^*_1)$ is a closed smooth $\G$-invariant submanifold of $M$. Also all points $p \in M \setminus (\pi^{-1}(B) \cup N)$ are noncritical points for the distance function $d_N$. Then the claim follows from standard methods in critical point theory.

\textit{Case 2.} Let $E_1 \neq \emptyset$. Set 
$$\Omega^*_2 := \{x \in \Omega^*_1 \mid d(x,E_1) \text{ is maximal }\}.$$
Then $\Omega^*_2$ is horizontally convex as well and $d_{\Omega^*_2}$ is regular on $M^* \setminus (B \cup \Omega^*_2)$. Let $E_2 \subset \Omega^*_2$ denote the nonlinear points of $\Omega^*_2$. Then again $E_2$ defines a boundary stratum of $\Omega^*_2$ if $E_2$ is nonempty.

\textit{Case 2.1.} Assume $E_2 = \emptyset$. Then we set $N = \pi^{-1}(\Omega^*_2$ and again the claim follows.

\textit{Case 2.2.} Assume $E_2 \neq \emptyset$. Then we consider $$\Omega^*_3 = \{x \in \Omega^*_2 \mid d(x,E_2) \text{ is maximal }\}$$
and the set $E_3$ of nonlinear points of $\Omega^*_3$.
\\

Iterating this argument after a finite number of steps we find $E_k = \emptyset$, since the dimension of $E_k$ increases at every step and we are done. \hfill$\square$
\\

Having finished the proof of theorem \ref{main} let us note that our arguments should apply as well for quotient spaces of Riemannian foliations or for orbifolds leading to analogous results.
\\

In the following we discuss some implications of theorem \ref{main}. For that we fix a submanifold $N \subset M$ as given by theorem \ref{main}.
\\

Every boundary stratum is a finite collection of faces of the boundary $\partial M^*$, by which we mean the type components of $M^*$ of codimension $1$. The codimension of the preimage of such a face $F$ is given by $\dim \mathsf K/\mathsf H + 1$, where $\mathsf K$ is a isotropy group of generic type of $F$ and $\mathsf H \subset \mathsf K$ is of principal type. Let a decomposition
\begin{align}\label{normal bundle}
M \setminus \pi{^-}(B) \cong \nu N
\end{align}
as in theorem \ref{main} be given, where $\nu N$ denotes the normal bundle of $N$. Analogously to the soul orbit theorem in \cite{wilking06} the connectedness of the inclusion map $N \hookrightarrow M$ is then restricted by a face of $B$ whose generic type is of minimal dimension.
\begin{cor}
Let $F$ be a face of $B$ of generic isotropy type $\mathsf K$ such that $\dim \mathsf K$ is minimal among all dimensions of generic isotropy groups of faces contained in $B$. Then the inclusion map $N \hookrightarrow M$ is $\dim \mathsf K/\mathsf H$-connected, where $\mathsf H \subset \mathsf K$ is of principal type.
\end{cor}
\begin{proof}
This follows from \eqref{normal bundle} together with $\dim \pi^{-1}(F) \leq \dim M - \dim \mathsf K/ \mathsf H - 1$ for every face $F \subseteq B$, compare the proof of the soul orbit theorem in \cite{wilking06}.
\end{proof}
Given that $\pi^{-1}(B)$ is a smooth submanifold of $M$ as well the situation becomes somewhat nicer. Then by the regularity of $d_N$ a gradientlike vectorfield $X$ with respect to $N$ on $M \setminus (\pi^{-1}(B) \cup N)$ can be constructed which is radial near $N$ and $\pi^{-1}(B)$. By the flow of $X$ we then obtain a diffeomorphism $\partial D(\pi^{-1}(B)) \cong \partial D(N)$, where $D(\pi^{-1}(B))$ and $D(N)$ denote the respective normal disc bundles of $\pi^{-1}(B)$ and $N$. Therefore we obtain the following corollary.
\begin{cor}\label{ddb}
Assume that $\pi^{-1}(B)$ is a smooth submanifold of $M$. Then $M$ is equivariantly diffeomorphic to the normal disc bundles of $\pi^{-1}(B)$ and $N$ glued together along their boundaries;
\begin{align}\label{here?}
M \cong D(\pi^{-1}(B)) \cup_\partial D(N).
\end{align}
\end{cor}
A manifold obtained as in \eqref{here?} is frequently called a double disc bundle. If we assume in the situation of this corollary that further $\G$ and $B$ are connected and $M$ is simply connected we obtain a mild bound on the dimension of $N$.
\begin{lemma}\label{codim}
Let the situation be as in corollary \ref{ddb} and assume further that $M$ is simply connected and $\G$ and $B$ are connected. Then $N$ has codimension greater than or equal to $2$ in $M$.
\end{lemma}
\begin{proof}
We consider the double disk bundle decomposition 
$$M = D(\pi^{-1}(B)) \cup_E D(N),$$
where we denote by $E$ the boundary of $D(N)$. Assume that $N$ has codimension $1$. Since $\G$ and $B$ are connected, it follows that $E = \partial D(N) \cong \partial D(\pi^{-1}(B))$ is connected as well. Therefore, the projection map $p : E \to N$ is a two fold covering map with connected total space. Hence $\pi_1(N)/p_*(\pi_1(E)) \cong \ZZ_2$ ($p_*$ denotes the morphism of fundamental groups induced by $p$). 
 
In contradiction to this we show that $\pi_1(N)/\pi_1(p)(E)$ is trivial using the van Kampen theorem: Let $q : E \to \pi^{-1}(B)$ denote the projection map and set $U := p_*(\pi_1(E)) \subseteq \pi_1(N)$. Then the following diagram commutes;
$$\begin{xy}
\xymatrix
{\pi_1(E) \ar^{p_*}[r] \ar_{q_*}[d] & \pi_1(N) \ar[d]^{[\ ]}\\
\pi_1(\pi^{-1}(B)) \ar_{0}[r] & \pi_1(N)/U.}
\end{xy}
$$
Here, $[\ ]$ denotes the quotient map (note that $U$ is normal in $\pi_1(N)$ since it has index $2$). By the van Kampen theorem $\pi_1(M)$ is the pushout of the maps $p_*$ and $q_*$. Therefore, there exists a morphism $h : \pi_1(M) \to \pi_1(N)/U$ such that the diagram
$$\begin{xy}
  \xymatrix{
      \pi_1(E) \ar[r]^{p_*} \ar[d]_{q_*}  &  \pi_1(N) \ar[d] \ar@/^/[ddr]^{[\ ]}  &  \\
      \pi_1(\pi^{-1}(B)) \ar[r] \ar@/_/[drr]_0  &  \pi_1(M) \ar[dr]^h  &  \\
      &  &  \pi_1(N)/U
  }
\end{xy}$$
commutes. Now, since $\pi_1(M)$ is trivial, it follows that the quotient map $[ \ ]$ is the $0$-map. Hence $\pi_1(N)/U$ is trivial.
\end{proof}


\section{Nonnegatively curved fixed point homogeneous manifolds and torus manifolds}\label{section fph}
In this section we state and reprove the main results about fixed point homogeneous actions on nonnegatively curved manifolds and nonnegatively curved torus manifolds from the authors dissertation \cite{spindeler14}, see also \cite{spindeler15}. Most of it is a reproduction of section 3.2.2 therein. However, since these results have not been peer-reviewed and they follow without further reference to \cite{spindeler14} from our theorem \ref{main} we include them here.
\\

We call an isometric action of a compact Lie group $\G$ on a complete Riemannian manifold $M$ fixed point homogeneous if its fixed point set $\operatorname{Fix}(\G)$ is nonempty and there exists a component $F$ of $\operatorname{Fix} (\G)$ which is a boundary component of $M^*$. Note that such a component is then given by every fixed point component of maximal dimension. Fixed point homogeneous manifolds of positive curvature first emerged in \cite{grove-searle94} and were later on studied in their own right in \cite{grove-searle97}. The techniques used there were first adapted to nonnegative curvature and in dimensions less or equal than $4$ in \cite{galaz-garcia12} and later to dimension $5$ in \cite{galaz-garcia-spindeler12}. In any dimension we now obtain the following result.
\begin{thm}\label{fph}
Let $\G$ act fixed point homogeneously on a compact nonnegatively curved Riemannain manifold $M$ and let $F$ denote a fixed point component of maximal dimension. Then there exists a closed smooth $\G$-invariant submanifold $N$ of $M$ such that $M$ is equivariantly diffeomorphic to the normal disc bundles of $F$ and $N$ glued together along their boundaries;
\begin{align*}
M \cong D(F) \cup_\partial D(N).
\end{align*}
\end{thm}
\begin{proof}
This follows immediately from corollary \ref{ddb}.
\end{proof}
A closed connected manifold $M$ is called rationally $\Omega$-elliptic if the total rational homotopy of the loop space, $\pi_*(\Omega M, \ast) \otimes \mathbb Q$, is finite dimensional. $M$ is called rationally elliptic if it is rationally $\Omega$-elliptic and simply connected.

If we consider a double disk bundle $M = D(F) \cup_\partial D(N)$, where $D(F)$ and $D(N)$ are disk bundles over closed manifolds $F$ and $N$, then $F$ is rationally $\Omega$-elliptic if and only $\partial D(F)$ is. Moreover, from \cite{grove-halperin87}, Corollary 6.1 it follows that a simply connected manifold which admits a double disk bundle decomposition is rationally $\Omega$-elliptic if and only if the boundary of one of the two disk bundles is rationally $\Omega$-elliptic. Therefore, from theorem \ref{fph} we obtain the following theorem.

\begin{thm}\label{elliptic}
Let $M$ be a closed simply connected fixed point homogeneous manifold of nonnegative curvature and $F$ be a fixed point component of maximal dimension. Then $M$ is rationally $\Omega$-elliptic if and only if $F$ is rationally $\Omega$-elliptic. 
\end{thm}
We conclude with an application of this results to nonnegatively curved torus manifolds.
\begin{definition}
A torus manifold is a smooth, connected, closed and orientable manifold $M$ of even dimension $2n$ admitting a smooth and effective action by the $n$-dimensional torus $\mathsf T^n$ with nonempty fixed point set. 
\end{definition}

\begin{thm}
Let $M$ be a closed and simply connected torus manifold equipped with an invariant metric of nonnegative curvature. Then $M$ is rationally elliptic.
\end{thm}
\begin{proof}
Let $\dim M = 2n$ and $\T^n$ act effectively and isometrically with nonempty fixed point set on $M$. Let $p_0 \in \operatorname{Fix}(\T^n)$ and consider the orthogonal action of $\T^n$ on $\sphere^{2n - 1} \subset T_{p_0}M$ induced by the slice representation. It was shown in \cite{grove-searle94}, Theorem 2.2,  that there exists a $1$-dimensional torus $\T^1_1 \subset \T^n$ acting fixed point homogeneous on $\sphere^{2n - 1}$. Hence $\T_1^1$ also acts fixed point homogeneous on $M$ and there exists a maximal fixed point component $F$ containing $p_0$. Consequently
\begin{align}\label{decaf}
M \cong D(F) \cup_\partial D(N).
\end{align}
for a smooth submanifold $N$ with $\dim N \leq 2n -2$, by Theorem \ref{ddb} and Lemma \ref{codim}. We claim that $F$ is simply connected:

\textit{case 1:} Assume $\dim N  \leq 2n - 3$. Then, by transversality, $\pi_1(F) = \pi_1 (M \setminus N) = \pi_1(M)$. Hence $\pi_1(F) = 0$.

\textit{case 2}:  Assume $\dim N = 2n - 2$. Let $E := \partial D(F) \cong \partial D(N)$. Then $N$ is orientable and there exists a $\T^1_2$-action on the normal bundle of $N$ obtained by orthogonally rotating the fibers. For a proof of this claim we refer to the proofs of Propositions 3.5 and 3.6 in \cite{galaz-garcia-spindeler12}. This action commutes with $\T^n$ and we obtain a smooth $\T^1_2$-action on $D(N) = E \times_{T_1^2} D^2$ which can be extended to $(E \times_{\T_1^1} D^2) \cup_E (E \times_{T_1^2} D^2) = D(F) \cup_E D(N) = M$.

Let $q \in \operatorname{Fix}\T^n$, $t \in \T^n$ and $g \in \T^1_2$. Then $t.(g.q) = g.(t.q) = g.q$. Hence $\operatorname{Fix} \T^n$ is invariant under $\T_2^1$. Since $\T^1_2$ is connected, and $\operatorname{Fix} \T^n$ is discrete, we see that $\operatorname{Fix} \T^n \subseteq \operatorname{Fix} \T_2^1$. It follows that $\T^2 := \T^1_1 \oplus \T^1_2$ acts on $M$ with $p_0 \in \operatorname{Fix} \T^2$. 

Consider the projections $f_1 : E \to E/\T_1^1 = F$ and $f_2 : E \to E/\T_2^1 = N$. From the homotopy sequences of this fibrations we obtain exact sequences
$$\dots \to \pi_1(\T_1^1) \xrightarrow{{i_1}_*} \pi_1(E) \xrightarrow{{f_1}_*} \pi_1(F) \to 1$$
and
$$\dots \to \pi_1(\T_2^1) \xrightarrow{{i_2}_*} \pi_1(E) \xrightarrow{{f_2}_*} \pi_1(N) \to 1,$$
where the maps $i_1$ and $i_2$ are the inclusions of the fibers over a given basepoint. Set $U_k = {i_k}_*(\pi_1(\T^1_k))$ for $k = 1,2$.
So $\pi_1(F) \cong \pi_1(E)/U_1$, $\pi_1(N) \cong \pi_1(E)/U_2$ and we have a commutative diagram
\begin{align}\label{diagram}
\begin{xy}
\xymatrix
{\pi_1(E) \ar^{{f_2}_*}[r] \ar_{{f_1}_*}[d] & \pi_1(N) \ar[d]\\
\pi_1(F) \ar[r] & \pi_1(E)/U_1U_2.}
\end{xy}
\end{align}
Here the lower map is given via $\pi_1(F) \cong \pi_1(E)/U_1 \to \pi_1(E)/U_1U_2$, and analogously for the map on the right. By \eqref{decaf} and the van Kampen theorem there exists a morphism $h : \pi_1(M) \to \pi_1(E)/U_1U_2$ making the following diagram commute:
$$\begin{xy}
  \xymatrix{
      \pi_1(E) \ar[r]^{{f_2}_*} \ar[d]_{{f_1}_*}  &  \pi_1(N) \ar[d] \ar@/^/[ddr]  &  \\
      \pi_1(F) \ar[r] \ar@/_/[drr]  &  \pi_1(M) \ar[dr]^h  &  \\
      &  &  \pi_1(E)/U_1U_2
  }
\end{xy}$$
Since all the maps in \eqref{diagram} are surjective it follows that $h$ is surjective as well. Since $\pi_1(M) = 0$, it follows that $\pi_1(E) = U_1U_2$. Hence, $\pi_1(E)$ is generated by the orbits $\T^1_1(q)$ and $\T^1_2(q)$ for a given point $q \in E$. Therefore, the map $\tau_q : \T^2 \to E$, $g \mapsto g.q$ induces a surjection ${\tau_q}_* : \pi_1(\T^2) \to \pi_1(E)$ for all $q \in E$. Consequently we obtain a surjection $(f_1 \circ \tau_q)_* : \pi_1(\T^2) \to \pi_1(F).$

Pick $q_0 \in E$ such that $f_1(q_0) = p_0$. Since $p_0 \in \operatorname{Fix} \T^2$, the map $f_1 \circ \tau_{q_0}$ is constant. Thus $(f_1 \circ \tau_{q_0})_* = 0$ and it follows that $F$ is simply connected.

Because $F$ is totally geodesic, $F$ also has nonnegative curvature. Further $\T^{n - 1} = \T^n/\T^1_1$ acts effectively on $F^{2n - 2}$ with nonempty fixed point set. So $F$ is a nonnegatively curved, simply connected Torus manifold as well.

Now the proof follows by induction on $n$ using theorem \ref{elliptic}.
\end{proof}
Finally we note that, using also the results of this chapter, Wiemeler shows in \cite{wiemeler14} that a compact and simply connected torus manifold admitting an invariant metric of nonnegative curvature is diffeomorphic to a quotient of a free linear torus action on a product of spheres.
\bibliographystyle{alpha}
\bibliography{qwb}

\end{document}